\documentclass[12pt]{amsart}
\usepackage[scale=.80]{geometry}
\usepackage{geometry}
\usepackage{graphicx}
\usepackage{bbm}
\usepackage{mathrsfs}
\theoremstyle{plain}

\usepackage{CJK}
\usepackage{amsfonts,amssymb,amsmath,mathrsfs,multirow,booktabs}

\usepackage{color,latexsym,amsfonts}
 \newtheorem{theorem}{Theorem}[section]
 \newtheorem{lemma}[theorem]{Lemma}
 \newtheorem{corollary}[theorem]{Corollary}
 \newtheorem{proposition}[theorem]{Proposition}
 \newtheorem{example}[theorem]{Example}
 \newtheorem{Definition}[theorem]{Definition}
 \newtheorem{condition}[theorem]{Condition}
\theoremstyle{remark}
\newtheorem{remark}[theorem]{Remark}

 \def\beqlb{\begin{eqnarray}}\def\eeqlb{\end{eqnarray}}
 \def\beqnn{\begin{eqnarray*}}\def\eeqnn{\end{eqnarray*}}

 \def\<{\langle}\def\>{\rangle}

 \def\eqref#1{{\rm(\ref{#1})}}

 \def\qed{\hfill$\Box$\medskip}

\def\<{\left<}\def\>{\right>}

\newfam\msbmfam\font\tenmsbm=msbm10\textfont
\msbmfam=\tenmsbm\font\sevenmsbm=msbm7
\scriptfont\msbmfam=\sevenmsbm

\def\<{\left<}\def\>{\right>}

\def\({\left(}\def\){\right)}

\title[Explosion in simple exchangeable fragmentation-coagulation processes]{ \bf  On the explosion of the number of fragments in the simple exchangeable fragmentation-coalescence processes 
}
\usepackage[pagebackref,colorlinks,citecolor=Plum,urlcolor=Periwinkle,linkcolor=DarkOrchid]{hyperref}
\newcommand{\ddr}{\mathrm{d}}
\usepackage[dvipsnames]{xcolor}
\usepackage{tikz}

\keywords{{$\Lambda$-coalescent}, {fragmentation}, {branching process}, {explosion}, {coming down from infinity}, {entrance boundary}, {regular boundary}, {continuous-time Markov chains}.}
\subjclass[2010]{60J80, 60J70, 60J90, 92D25}
\begin{document}
\maketitle
\centerline{\large  Cl\'ement Foucart \footnote{Universit\'e  Sorbonne Paris Nord and Paris 8, Laboratoire Analyse, G\'eom\'etrie $\&$ Applications, UMR 7539. Institut Galil\'ee, 99 avenue J.B. Cl\'ement, 93430 Villetaneuse, France, foucart@math.univ-paris13.fr} and Xiaowen Zhou \footnote{Department of Mathematics and Statistics, Concordia University, 1455 De Maisonneuve Blvd. W., Montreal, Canada, xiaowen.zhou@concordia.ca}
}
\begin{center}
\today

\end{center}
\begin{abstract}
We consider the exchangeable fragmentation-coagulation (EFC) processes, where the coagulations are multiple and not simultaneous, as in a $\Lambda$-coalescent, and the fragmentations dislocate at finite rate an individual block into sub-blocks of infinite size. Sufficient conditions are found for the block-counting process to explode (i.e. to reach $\infty$)  or not and for infinity to be an exit boundary or an entrance boundary. In a case of regularly varying fragmentation and coagulation mechanisms, we find regimes where the boundary $\infty$ can be either an exit, an entrance or a regular boundary.  In the latter regular case, the EFC process leaves instantaneously the set of partitions with an infinite number of blocks and returns to it immediately. Proofs are based on a new sufficient condition of explosion for positive continuous-time Markov chains, which is of independent interest.
\end{abstract}

\section{Introduction}

Stochastic processes describing both coalescence and fragmentation are ubiquitous in scientific disciplines such as astrophysics, chemistry, genetics, or population dynamics. We refer for instance to Aldous \cite{Aldous} for a review of applications. Berestycki in \cite{Berestycki04} has characterized and studied the class of exchangeable fragmentation-coalescence processes, called EFC processes for short.
An EFC process is a process $(\Pi(t),t\geq 0)$,  for which at any time $t\geq 0$, $\Pi(t)$ stands for a collection, possibly finite, of disjoint subsets, called ``fragments" or ``blocks", $(\Pi_1(t),\Pi_2(t),\cdots)$,  covering the set of positive integers $\mathbb{N}:=\{1,2,\cdots\}$ i.e. $\cup_{i\geq 1}\Pi_i(t)=\mathbb{N}$. The process is exchangeable in the sense that for any time $t\geq 0$, the random partition $\Pi(t)$ of $\mathbb{N}$ has a law invariant under the action of permutations that only change finitely many integers. Last but not least, the evolution of the process is two-fold. Blocks can merge, as in an exchangeable coalescent, and can fragmentate as in an homogeneous exchangeable fragmentation.

This article considers the problem of classifying the nature of the boundary $\infty$ of the continuous-time Markov chains arising as the functional of the number of blocks in EFC processes. Our main goal is to study the phenomenon of explosion in the number of blocks. 

When fragmentation occurs at an infinite rate, the number of blocks is infinite at almost all times, see \cite[Theorem 12]{Berestycki04}. We shall not consider this case and will focus here on the class of EFC processes with fragmentations occurring at a finite rate, and in which neither simultaneous fragmentations nor simultaneous  multiple coagulations can occur. We shall assume moreover that the fragmentations cannot dislocate blocks into singletons. These processes, called \textit{simple} EFC processes, can therefore be seen as a generalisation of $\Lambda$-coalescents, as defined by Pitman \cite{Pitman99} and Sagitov \cite{MR1742154}, in which ``simple" fragmentations are incorporated. More precisely, each infinite block is splitted into $k+1$ sub-blocks of infinite size, (thus, creating $k$ new blocks), at rate say $\mu(k)$, independently of each other, where $\mu$ is a finite positive measure on $\mathbb{N}\cup\{\infty\}$. At the level of the number of blocks, the fragmentation is therefore nothing but a discrete branching process with no death, whose offspring measure is $\mu$. We shall call $\mu$ the \textit{splitting} measure. 

EFC processes arise for instance when studying the frequency of a disadvantaged allele in certain Wright-Fisher models with selection, see Gonz\'alez-Casanova and Span\`o \cite{GonzalesSpano} and the references therein. In terms of population models, fragmentations can be seen as reproduction events and coalescences as negative interactions between individuals in the population. This link is mentioned in  Lambert \cite[Section 2.3]{Lambert:2005bq}, where it is shown that if there are only binary coagulation events, the block-counting process of the associated simple EFC process has the same law as a discrete logistic branching process. More generally, $\Lambda$-coalescences can be interpreted as a competition term between multiple individuals. This point of view was chosen for instance in Gonz\'alez-Casanova et al. \cite{Gonzalesetal}. Some continuous-state space models are also closely related to EFC processes. We refer the reader for instance to Bansaye et al. \cite{Bansayeetal} and Foucart \cite{FoucartEJP}. We wish to mention the work of Wagner \cite{Wagner} where the phenomenon of explosion is studied for a different family of coagulation-fragmentation particle systems. See also Bertoin and Kortchemski \cite[Section 5.4]{BertoinKortchemski} where scaling limits of some EFC processes are studied.

A remarkable feature of $\Lambda$-coalescent processes lies in the fact that under certain conditions on the coalescence, the process, started from a partition with infinitely many blocks, can instantaneously enter the set of partitions with a finite  number of blocks. This phenomenon, called \textit{coming down from infinity}, has been deeply studied in the 2000s, see \cite{Schweinsberg00}, \cite{Beres10}  and \cite{limic2015}. In particular, Schweinsberg \cite{Schweinsberg00} has found a necessary and sufficient condition on the measure $\Lambda$ for the coming down from infinity. In the pure coalescent framework, the block-counting process has decreasing sample paths and when it starts from infinity and leaves it, it stays finite at any further time. In Feller's terminology, see \cite{Feller} and e.g. Anderson \cite[Chapter 8, page 262]{Anderson}, $\infty$ is said to be an \textit{entrance} boundary. 

In a symmetric way, without coalescences, when the process starts from a partition with finitely many blocks, fragmentations into finitely many sub-blocks may accumulate and push the number of blocks to $\infty$ in finite time, which is referred to as the phenomenon of \textit{explosion}. It is also well-known that $\infty$ is an absorbing state for branching processes, so that if the process of pure fragmentation explodes then it stays infinite at any further time. In Feller's terminology, $\infty$ is said to be an \textit{exit} boundary.

When both fragmentations and coalescences are taken into account, sample paths of the block-counting process are not monotone anymore, and some new phenomena may arise. For instance,  when the pure coalescent part does come down from infinity, fragmentations may or may not prevent the coming down from infinity of the EFC process.  When the pure fragmentation explodes, it is also natural to ask whether or not the coalescent part will prevent explosion. To the best of our knowledge, only few results in this direction are known for the moment.

An important step in the understanding of the possible behaviors at $\infty$ of EFC processes, has been recently made by Kyprianou et al. in \cite{kyprianou2017}. They study there the ``fast" fragmentation-coalescence process, in which coagulations are binary, as in a Kingman coalescent, and fragmentation dislocates at a constant rate, any individual block into its constituent elements (which creates infinitely many singleton blocks, and causes an infinite jump of the number of blocks). In \cite{kyprianou2017}, a phase transition is found between a regime for which the boundary is an \textit{exit} and a regime where the boundary $\infty$ is \textit{regular}, namely the block-counting process leaves and returns to $\infty$ almost-surely. In this latter regime, it is also shown in \cite{kyprianou2017} that the boundary $\infty$ is regular \textit{for itself}. That is to say, when started from a partition with infinitely many blocks, the partition-valued process leaves the set of partitions with infinitely many blocks and returns to it instantaneously. This leads to many open questions for less extreme  mechanisms of fragmentation and coalescences. It is natural for instance to wonder if the boundary $\infty$ can be regular for other EFC processes than  the  ``fast" EFC process.

The class of simple EFC processes with general $\Lambda$-coalescences has been recently studied in \cite{cdiEFC}. Let $(\Pi(t),t\geq 0)$ be a simple EFC process. The block-counting process, denoted by $(\#\Pi(t),t\geq 0)$, has the following infinitesimal dynamics when it evolves in $\mathbb{N}$. Let $n\in \mathbb{N}$.
\begin{itemize}
\item \textit{Coalescences}: for any $2\leq k\leq n$, it jumps  from $n$ to $n-k+1$ at rate $\binom{n}{k}\lambda_{n,k}$, with \[\lambda_{n,k}:=\int_{[0,1]}x^{k}(1-x)^{n-k}x^{-2}\Lambda(\ddr x).\]
\item \textit{Fragmentations}: for any $k\in \mathbb{N}\cup \{\infty\}$,
it jumps from $n$ to $n+k$, at rate $n\mu(k)$.
\end{itemize}
Unlike the fast EFC process, when the fragmentations cannot dislocate a block into infinitely many sub-blocks, i.e. $\mu(\infty)=0$, we cannot immediately deduce from the dynamics above whether the boundary $\infty$ can be reached or not. The question of accessibility of $\infty$ (namely explosion) in simple EFC processes was left unaddressed in \cite{cdiEFC}.
The first purpose of this article is to shed some light on the cases where fragmentations and $\Lambda$-coalescences together can allow the process to explode or not.

The coming down from infinity of simple EFC processes has been studied in \cite{cdiEFC}. A phase transition between a regime in which a simple EFC process, started from an exchangeable random partition with infinitely many blocks, comes down from infinity and one in which it stays infinite, is established in \cite[Theorem 1.1]{cdiEFC}. Combining this result, recalled in Section \ref{CDIEFC}, and our conditions for explosion/non-explosion,  we will find sufficient conditions on $\Lambda$ and $\mu$ for the boundary  $\infty$ to be either an exit or an entrance, see Theorem \ref{suffcondpropexit} and Theorem \ref{suffcondpropentrance} respectively. We study in details the nature of the boundary $\infty$ in two cases of regularly-varying coalescence and fragmentation mechanisms, see Theorem \ref{stablefragtheorem} and Theorem \ref{logcoal}. 

In particular, we shall see in Theorem \ref{stablefragtheorem}, that when the coalescence and splitting measures satisfy 
\begin{equation}\label{regularlambda}
\Lambda(\ddr x)=f(x)\ddr x,\text{ for } x\in [0,x_0] \text{ with } f(x)x^{\beta}\underset{x\rightarrow 0+}{\longrightarrow}c 
\end{equation}
\text{ and } 
\begin{equation}\label{regularmu}
\mu(n)n^{-(1+\alpha)}\underset{n\rightarrow \infty}{\longrightarrow} b
\end{equation}
for some $x_0\in (0,1]$, $\alpha,\beta\in (0,1)$ and $b,c>0$, then indeed the boundary $\infty$ can be regular.
When $\alpha+\beta =1$, new phase transitions are found between regimes where $\infty$ is regular, an exit or an entrance. See the forthcoming Figure \ref{graph} for a summary of the possible behaviors. We will also show that when the boundary is regular, it is regular for itself. 


The study of the explosion is based on a new sufficient condition for explosion of general continuous-time Markov chains, see Theorem \ref{CTMCtheorem} and Corollary \ref{genexplosionmaincor}. Theorem \ref{CTMCtheorem}  is based on estimates for the first passage times above large levels. The main difficulties  that arise when $\Lambda$-coalescences are allowed come from the fact that downwards jumps can have a size of the same order  as the level of the process prior the jump. We shall see how to deal with those large jumps, and how to measure both coagulation and fragmentation strengths in order to apply our new condition of explosion.

The article is organized as follows. Our main results are stated in Section \ref{results}. In Section \ref{preliminaries}, we provide more background on EFC processes. We briefly recall their Poisson construction, as well as some important properties of the functional of the number of blocks. We then recall some results about the coming down from infinity as well as results on explosion of pure branching processes. Our new condition for explosion is stated and established in Section \ref{suffcondexplctmc}. Proofs of the main results are given in Section \ref{proofs}.


\textbf{Notation:} For any integers $n\leq m$, we denote by $[|n,m|]$ the interval of integers $\{n,\cdots, m\}$.  We shall say that a property $P(n)$ on the integer $n$ is true for large enough $n$, when there exists $n_0\in \mathbb{N}$ such that $P(n)$ holds for all $n\geq n_0$. For any $a\in \mathbb{R}_+\cup \{\infty\}$ and any $f$ a positive Borel function defined on $(a-\epsilon,a)$ for some $\epsilon>0$, we denote the integrability of $f$ at $a$ by $\int^{a}f(x)\ddr x<\infty$. The left and right limit at $a\in \mathbb{R}$ of $f$ are denoted by $\underset{x\rightarrow a^{-}}\lim f(x)$ and $\underset{x\rightarrow a^{+}}\lim f(x)$, respectively. Lastly, for any positive functions $f$ and $g$ well-defined in a neighbourhood of $a$, we use Landau's notation for asymptotic equivalence, namely we set $f(x)\underset{x\rightarrow a}{\sim} g(x)$ when $\frac{f(x)}{g(x)}\underset{x\rightarrow a}{\longrightarrow} 1$.

\section{Preliminaries}\label{preliminaries}
\subsection{Exchangeable fragmentation-coalescence processes}\label{EFCbasics}\

We denote by $\mathcal{P}_\infty$ the space of partitions of $\mathbb{N}$. By convention, 
any partition $\pi$ of $\mathbb{N}$ is represented by the sequence (possibly finite) of its non-empty blocks $(\pi_i,i\geq 1)$ ordered by their least element. For any $n\in \mathbb{N}$, we denote by $\pi_{|[n]}$ the partition restricted to $[n]:=\{1,\cdots,n\}$: namely $\pi_{|[n]}=(\pi_1\cap [n],\pi_{2}\cap[n],\cdots)$. The space $\mathcal{P}_\infty$ is endowed with the metric $d$, defined by $d(\pi,\pi'):=\max\{n\geq 1; \pi_{|[n]}=\pi'_{|[n]}\}^{-1}$. For any partition $\pi\in \mathcal{P}_\infty$, we denote by $\#\pi$, its number of non-empty blocks. By convention, if $\#\pi<\infty$, then we set $\pi_{j}=\emptyset$ for any $j\geq \#\pi+1$. 

Exchangeable fragmentation-coalescence processes are Feller processes with state-space $\mathcal{P}_\infty$. It is established in \cite{Berestycki04} that they are characterized in law by two $\sigma$-finite exchangeable measures on $\mathcal{P}_\infty$, $\mu_{\mathrm{Coag}}$ the measure of coagulation and $\mu_{\mathrm{Frag}}$, that of fragmentation.  We refer the reader to Berestycki's article for the general form that can take those measures, as well as the integrability conditions they must satisfy.

We briefly recall now the Poisson construction of EFC processes with given coagulation and fragmentation measures. See \cite{Berestycki04} for details on this construction.

Consider two independent Poisson Point Processes $\mathrm{PPP}_C$ and $\mathrm{PPP}_F$ respectively on $\mathbb{R}_+\times \mathcal{P} _\infty$ and $\mathbb{R}_+\times \mathcal{P} _\infty\times \mathbb{N}$ with intensity $\ddr t\otimes \mu_{\mathrm{Coag}}(\ddr \pi)$ and  $\ddr t\otimes \mu_{\mathrm{Frag}}(\ddr \pi)\otimes \#$ where $\#$ denotes the counting measure on $\mathbb{N}$. Let $\pi$ be an exchangeable random partition independent of $\mathrm{PPP}_F$ and $\mathrm{PPP}_C$. For any $m\geq 1$, define the process $(\Pi^{m}(t),t\geq 0)$ as follows: $\Pi^{m}(0)=\pi_{|[m]}$ and 
\begin{align*}
\Pi^{m}(t)&=\mathrm{Coag}\left(\Pi^{m}(t-),\pi^{c}_{|[m]}\right) \text{ if } (t,\pi^{c}) \text{\,\, is an atom of } \mathrm{PPP}_C, \\
\Pi^{m}(t)&=\mathrm{Frag}\left(\Pi^{m}(t-),\pi^{f}_{|[m]},j\right) \text{ if } (t,\pi^{f},j) \text{\,\, is an atom of } \mathrm{PPP}_F,
\end{align*}
where for any partitions $\pi$, $\pi^{c}$, $\pi^{f}$
\begin{align*}
\mathrm{Coag}(\pi,\pi^{c})&:=\{\cup_{j\in \pi^{c}_i}\pi_j, i\geq 1\} \text{ and } \mathrm{Frag}(\pi,\pi^{f},j):=\{\pi_j\cap\pi^{f}_i, i\geq 1;\ \pi_{\ell}, \ell\neq j\}^{\downarrow}
\end{align*}
where $\{\cdots\}^{\downarrow}$ means that we reorder the blocks by their least elements. See Bertoin's book \cite{coursbertoin} for fundamental properties of the operators $\mathrm{Coag}$ and $\mathrm{Frag}$.  The processes $(\Pi^{m}(t),t\geq 0)_{m\geq 1}$ are compatible in the sense that for any $m\geq n\ge 1$, $$(\Pi^{m}(t)_{|[n]},t\geq 0)=(\Pi^{n}(t),t\geq 0).$$ This ensures the existence of a process $(\Pi(t),t\geq 0)$ on $\mathcal{P}_\infty$ such that for all $m\geq 1$
\[(\Pi(t)_{|[m]},t\geq 0)=(\Pi^{m}(t),t\geq 0).\]
The process $(\Pi(t),t\geq 0)$ is an exchangeable EFC process started from the exchangeable random partition $\Pi(0)=\pi$. Among other results, Berestycki has shown that the $\mathcal{P}_\infty$-valued process $(\Pi(t),t\geq 0)$ is a c\`adl\`ag Feller process. 

In this article, we will focus on EFC processes in which there are no multiple simultaneous mergings, as in a $\Lambda$-coalescent, fragmentations occur at finite rate and dislocate any blocks into sub-blocks of infinite size. Formally, the coagulation measure charges partitions with only one non-singleton block, and the fragmentation measure  $\mu_{\mathrm{Frag}}$ on $\mathcal{P}_\infty$ has finite mass, i.e. $\mu_{\mathrm{Frag}}(\mathcal{P}_\infty)<\infty$, and only charges partitions whose blocks are all of infinite size.
According to \cite[Proposition 2.11]{cdiEFC}, if $(\Pi(t),t\geq 0)$ is a simple EFC process then the process $(\#\Pi(t),t\geq 0)$ has right-continuous sample paths in $\bar{\mathbb{N}}=\mathbb{N}\cup\{\infty\}$, the one-point compactification of $\mathbb{N}$. It is important to notice that the map $\pi\in \mathcal{P}_\infty \mapsto \#\pi$ is neither continuous with respect to the metric $d$, nor injective, see e.g. \cite[Remark 2.7]{cdiEFC}. This entails in particular that the process $(\#\Pi(t),t\geq 0)$ is not clearly Markovian. 

Following Bertoin's terminology, see \cite[Chapter 2.3]{coursbertoin}, we call any exchangeable partition with no singleton blocks (namely with no \textit{dust}) \textit{a proper partition}. A simple application of the paint-box construction of exchangeable partitions allows one to construct a proper initial random exchangeable partition $\pi$ with $\#\pi=n$ almost-surely for any $n\in \bar{\mathbb{N}}$. Since by the assumption, simple EFC processes have homogeneous fragmentations and the fragmentation measure has its support on partitions containing no singletons, the simple EFC process $(\Pi(t),t\geq 0)$, when started from a proper partition, stays proper at any time. We refer to Bertoin \cite{Bertoin2003} for the fact that there is no formation of dust in homogeneous fragmentation processes.

Denote by $\tau_\infty^{+}:=\inf\{t>0; \#\Pi(t-)=\infty\}$ the first explosion time of $(\#\Pi(t),t\geq 0)$. According to \cite[Proposition 2.11]{cdiEFC}, for any $n\in \mathbb{N}$, the process $(\#\Pi(t),t <\tau_\infty^{+})$ started from $\#\Pi(0)=n$ is a Markov process, and we denote its law by $\mathbb{P}_n$. The dynamics of $(\#\Pi(t),t<\tau_{\infty}^{+})$ can be explained from those of $(\Pi(t),t\geq 0)$ as follows.
\vspace*{1mm}

\textit{Coalescence.} Associate to each atom of $\mathrm{PPP}_C$, $(t,\pi^{c})$, a sequence of random variables $(X_i,i\geq 1)$ such that $X_i=1$ if $\{i\}\notin \pi^{c}$ and $X_i=0$ if $\{i\}\in \pi^{c}$. The random variables $(X_i,i\geq 1)$ are mixtures of i.i.d Bernoulli random variables with parameter $x$ whose ``intensity" is of the form $x^{-2}\Lambda(\ddr x)$ for some finite measure $\Lambda$ on $[0,1]$. Upon the arrival of an atom of $\mathrm{PPP}_C$,  given $\#\Pi(t-)=n$, all blocks whose index $j\in [n]$ satisfies $X_{j}=1$ are merged. Given the parameter $x$ of the $X_i$'s, the number of blocks that merge at time $t$ has a binomial law with parameters $(n,x)$. Therefore, for any $k\in [|2,n|]$ the jump
\begin{equation}\label{multiplejump} \#\Pi(t)=\#\Pi(t-)-(k-1). \end{equation}
has rate $\binom{n}{k}\lambda_{n,k}$ where we recall $\lambda_{n,k}:=\int_{[0,1]}x^{k}(1-x)^{n-k}x^{-2}\Lambda(\ddr x)$.
Binary coalescences are hidden in the description above. They are governed by the Kingman parameter $\Lambda(\{0\})=:c_{\mathrm{k}}\geq 0$. We shall always assume that $\Lambda$ has no mass at $1$, so that it is impossible for all the blocks to coagulate simultaneously at once.
\vspace*{1mm}

\textit{Fragmentation}. Associate to each atom of $\mathrm{PPP}_F$, $(t,\pi^{f},j)$, the random variable $k:=\#\pi^{f}-1$. This provides a Poisson point process on $\mathbb{R}_{+}\times \bar{\mathbb{N}}\times \mathbb{N}$ with intensity $dt\otimes \mu \otimes \#$, where $\mu$ is the image of $\mu_{\mathrm{Frag}}$ by the map $\pi\mapsto \#\pi-1$.  Upon the arrival of an atom $(t,\pi^{f},j)$, given $\#\Pi(t-)=n$, if $j\leq n$, then the
$j^{\mathrm{th}}$-block is fragmentated into $k+1$ blocks. Therefore, at time $t$,
\begin{equation}\label{fragmentation} \#\Pi(t)=\#\Pi(t-)+k.
\end{equation}
Since there are $n$ blocks at time $t-$, the total rate at which a jump of the form \eqref{fragmentation} occurs is  $n\mu(k)$ for any $k\in \bar{\mathbb{N}}$.

The generator of $(\#\Pi(t),t\geq 0)$ acts on functions on $\mathbb{N}$ as follows:
\begin{equation}\label{generator}  \mathcal{L}g:=\mathcal{L}^{c}g+\mathcal{L}^{f}g
\end{equation}
with for $n\in \mathbb{N}$
\begin{equation*} \mathcal{L}^{c}g(n):=\sum_{k=2}^{n}\binom{n}{k}\lambda_{n,k}[g(n-k+1)-g(n)] \text{ and }
\mathcal{L}^{f}g(n):=n\sum_{k=1}^{\infty}\mu(k)[g(n+k)-g(n)]
\end{equation*}
where $\mathcal{L}^{c}g(n)$ vanishes if $n=1$.

When $\mu(\infty)=0$, the generator $\mathcal{L}$ is conservative and $\infty$ cannot be reached by a single jump. However, it might still happen that infinitely many fragmentations have accumulated and pushed the sample path to ``reach"  $\infty$. Following the usual terminology for continuous-time Markov chains (CTMCs for short), we call this event \textit{explosion}. 
The question whether explosion occurs or not, requires a deep study of the generator $\mathcal{L}$. This is the main goal of the article and from now on we shall always assume $\mu(\infty)=0$. 

It is important to notice that since the partition-valued process $(\Pi(t),t\geq 0)$ has an infinite lifetime (namely, is defined at any time $t$), the process $(\#\Pi(t),t\geq 0)$ is also well-defined at any time $t$, as a process evolving in $\bar{\mathbb{N}}$, and typically is defined after explosion (if explosion occurs). Last, in view of the possible jumps of the block-counting process, if $\mu(\mathbb{N})>0$ then $\mathbb{N}$ is a communication class for the process $(\#\Pi(t),t\geq 0)$. Indeed, let $n_0$ and $n_1$ be two integers, if $n_1<n_0$, the process started from $n_0$ can reach $n_1$ by a coalescence of $n_0-n_1+1$ blocks. If $n_1>n_0$ and there is $k\in \mathbb{N}$ such that $\mu(k)>0$, then after $n_1$ jumps of size $k$, the process started at $n_0$ has reached the state $n_0+kn_1$, and from the latter can reach $n_1$ by a coalescence involving $n_0+(k-1)n_1+1$ blocks.
\begin{remark}\label{Markovatinfinity}
Since the map $\pi\mapsto \#\pi$ is not continuous with respect to the metric $d$, the Markov property of $(\#\Pi(t),t\geq 0)$ at any time $t$ such that $\#\Pi(t)=\infty$ is not straightforward. Indeed, the law of $(\#\Pi(t+s),s\geq 0)$ could depend on $\Pi(t)$ and not only on the fact that $\#\Pi(t)=\infty$. We refer to Kyprianou et al. \cite[Lemma 3.4]{kyprianou2017}, see also \cite[Remark 2.14]{cdiEFC} for more details. We stress that our proofs only make use of the Markov properties of the processes $(\Pi(t),t\geq 0)$ and $(\#\Pi(t),t<\tau_\infty^{+})$. The Markov property of the block-counting process actually holds true and is established by a duality relationship with $\Lambda$-Wright-Fisher process with selection, see \cite[Theorem 3.7]{WFselection}.
\end{remark}

Later on, we shall also be interested in the process $(\#\Pi(t),t\geq 0)$ when $(\Pi(t),t\geq 0)$ is started from an  exchangeable partition with infinitely many blocks. We recall here how to define on the same probability space as $(\Pi(t),t\geq 0)$, a monotone coupling of $(\Pi(t),t\geq 0)$ in the first initial blocks, when all blocks of $\Pi(0)$ are infinite. 
%
Assume $\Pi(0)$ proper and $\#\Pi(0)=\infty$ a.s. Let $n\in \bar{\mathbb{N}}$. Set $(\Pi^{(n)}(t),t\geq 0)$ the process  started from $\Pi^{(n)}(0):=\{\Pi_1(0),\cdots, \Pi_n(0)\}$ contructed from $\mathrm{PPP}_C$ and $\mathrm{PPP}_F$ as follows:
\begin{align*}
\Pi^{(n)}(t)&=\mathrm{Coag}\left(\Pi^{(n)}(t-),\pi^{c}\right) \text{ if } (t,\pi^{c}) \text{ is an atom of } \mathrm{PPP}_C,\\
\Pi^{(n)}(t)&=\mathrm{Frag}\left(\Pi^{(n)}(t-),\pi^{f},j\right) \text{ if } (t,\pi^{f},j)\, \text{ is an atom of } \mathrm{PPP}_F.
\end{align*}
Loosely speaking, the process $(\Pi^{(n)}(t),t\geq 0)$ follows the coalescences and fragmentations in the first $n$ initial blocks of $\Pi$. We refer to \cite[Lemma 3.3]{cdiEFC} for details on the Poisson construction. 

The following lemma will play a crucial role in our proofs. See \cite[Lemma 3.4]{cdiEFC}.
\begin{lemma}\label{monotonicity} Assume $\Pi(0)$ proper. Almost-surely for all $n\geq 1$ and all $t\geq 0$, $$\#\Pi^{(n)}(t)\leq \#\Pi^{(n+1)}(t)$$ and $\underset{n\rightarrow \infty}{\lim} \#\Pi^{(n)}(t)=\#\Pi(t)$ a.s.  Letting $\tau_\infty^{+}$ be the first explosion time of $(\#\Pi^{(n)}(t),t\geq 0)$, then for all $n\in \mathbb{N}$, $(\#\Pi^{(n)}(t),t< \tau_\infty^{+})$ has the same law as $(\#\Pi(t),t< \tau_\infty^{+})$ when $\#\Pi(0)=n$.
\end{lemma}
\begin{remark}\label{properremark}
This is necessary that the initial partition $\Pi(0)$  is proper, namely that each of its block is infinite, for $(\#\Pi^{(n)}(t),0\leq t< \tau_\infty^{+})$ to have the same law as $(\#\Pi(t),0\leq t< \tau_\infty^{+})$ started from $n$. Indeed, if a fragmentation event, at a time, say $t>0$, involves a block at time $t-$ which is finite, then the number of sub-blocks created after dislocating this block would depend on the block shape at time $t-$, and the Markov property would be lost.  See Proposition 2.11 in \cite{cdiEFC} and its proof.
\end{remark}

\subsection{Coming down from infinity for $\Lambda$-coalescent processes}\label{CDIcoal} \

Recall the Poisson description given in Section \ref{EFCbasics} and consider a pure $\Lambda$-coalescent process $(\Pi(t),t\geq 0)$. 
Pitman \cite[Proposition 23]{Pitman99} has established a zero-one law for the coming down from infinity of $\Lambda$-coalescents. Under the assumption $\Lambda(\{1\})=0$, when started from a partition with infinitely many blocks, either the process comes down from infinity almost-surely, or it stays infinite:
\begin{equation}\label{dichotomy}\exists t>0, \#\Pi(t)<\infty \text{ a.s. or } \forall t\geq 0, \#\Pi(t)=\infty \text{ a.s}.\end{equation}
Pitman \cite{Pitman99} has also shown that when a $\Lambda$-coalescent comes down from infinity, then provided $\Lambda(\{1\})=0$, it does it instantaneously a.s. that is to say if we set $\tau_\infty^{-}:=\inf\{t>0; \#\Pi(t)<\infty\}$, then $\tau_\infty^{-}=0$ a.s. The following necessary and sufficient condition for coming down from infinity of $\Lambda$-coalescents was discovered by Schweinsberg \cite{Schweinsberg00}.  Define for any  $n\geq 2$,
\begin{equation}\label{phi}
\Phi(n):=\sum_{k=2}^{n}(k-1)\binom{n}{k}\lambda_{n,k}. 
\end{equation} The $\Lambda$-coalescent $(\Pi(t),t\geq 0)$ comes down from infinity if and only if
\begin{equation} \label{Schweinsberg} \sum_{n=2}^{\infty}\frac{1}{\Phi(n)}<\infty \qquad \text{ (Schweinsberg's condition).} \end{equation}
The map $\Phi$ will play an important role in the sequel and we recall some of its properties.
For any $n\geq 2$, $\Phi(n)$ represents the rate of the total reduction of the block-counting process $(\#\Pi(t),t\geq 0)$, when it starts from $n$. We stress first that simple binomial calculations entail that for any $n\geq 2$
\begin{equation}\label{phi2}
\Phi(n)=\frac{c_{\mathrm{k}}}{2}n(n-1)+\int_{]0,1[}\left((1-x)^{n}+nx-1\right)x^{-2}\Lambda(\ddr x).
\end{equation}
One can also check that the map $n\mapsto \Phi(n)/n$ is non-decreasing and by the inequalities \[(1-x)^{n}+nx-1\leq e^{-nx}+nx-1\leq \frac{n^2}{2}x^2\]
for any $x\in [0,1]$, we see that $\Phi(n)\leq \Psi(n)\leq \frac{\Lambda([0,1])}{2}n^2$ for all $n\geq 2$, with
\begin{equation}\label{psi}\Psi(n)=\frac{c_\mathrm{k}}{2} n^{2}+\int_{0}^{1}(e^{-nx}-1+nx)x^{-2}\Lambda(\ddr x).
\end{equation}

We mention also the equivalence $\Phi(n)\underset{n\rightarrow \infty}{\sim} \Psi(n).$
We refer for these properties of the function $\Phi$ to Berestycki's book \cite[Chapter 4]{Beresbook} and Limic and Talarczyk \cite[Lemma 2.1]{limic2015}. In particular, when the measure $x^{-2}\Lambda(\ddr x)$ is regularly varying near $0$, the Tauberian theorem and the equivalence above ensure that $\Phi$ is regularly varying at $\infty$. For instance, if $\Lambda$ satisfies \eqref{regularlambda} then $\Phi(n)\underset{n\rightarrow \infty}{\sim} dn^{1+\beta}$ with $d:=c\frac{\Gamma(1-\beta)}{\beta(\beta+1)}>0$ where $\Gamma$ is the Gamma function.
This covers for instance the Beta-coalescent, see e.g. \cite{Beresbook}.

\subsection{Coming down from infinity for simple EFC processes}\label{CDIEFC}\

The coming down from infinity of simple EFC processes, namely the possibility to visit partitions with finitely many blocks, when started from a partition with an infinite number of blocks, has been studied in \cite{cdiEFC}. As noticed in Berestycki \cite{Berestycki04} and in \cite{cdiEFC}, the dichotomy \eqref{dichotomy} also holds when fragmentations are added. 
Similar to a pure $\Lambda$-coalescent (under the assumption $\Lambda(\{1\})=0$), an EFC either comes down from infinity instantaneously or stays infinite.
\begin{proposition}[Lemma 2.5 in \cite{cdiEFC}]\label{instantaneouscdi} Under the assumptions  of Theorem \ref{cditheorem}. Let $\tau_\infty^{-}:=\inf\{t>0; \#\Pi(t)<\infty\}$. Then $\mathbb{P}(\tau_{\infty}^{-}=0)=1$ or $\mathbb{P}(\tau_{\infty}^{-}=\infty)=1$.
\end{proposition}

Set the two following parameters
\[\theta^{\star}:=\underset{n\rightarrow \infty}{\limsup}\sum_{k=1}^{\infty}\frac{n\bar{\mu}(k)}{\Phi(n+k)}\in [0,\infty] \text{ and } \theta_{\star}:=\underset{n\rightarrow \infty}{\liminf}\sum_{k=1}^{\infty}\frac{n\bar{\mu}(k)}{\Phi(n+k)}\in [0,\infty].\]
\begin{theorem}[Theorem 1.1 in \cite{cdiEFC}]\label{cditheorem} Let $(\Pi(t),t\geq 0)$ be a simple EFC process started from an exchangeable random partition such that $\#\Pi(0)=\infty$.

Assume $\sum_{n=2}^{\infty}\frac{1}{\Phi(n)}<\infty$.
\begin{itemize}
\item If $\theta^{\star}<1$, then $(\Pi(t),t\geq 0)$ comes down from infinity a.s.
\item If $\theta_{\star}>1$, then $(\Pi(t),t\geq 0)$ stays infinite a.s.
\end{itemize}
\end{theorem}

The parameters $\theta^{\star}$ and $\theta_\star$ are somewhat intricate since both arguments $n$ and $k$ are not separated in the sum. However they coincide in many regular cases (if so, we denote the value by $\theta$) and can be computed explicitely when both  $\mu$ and $\Phi$ have regular variations.


\begin{proposition}[Regularly-varying cases, Proposition 1.6 in \cite{cdiEFC}]\label{regularcdi}\

If $\Phi(n)\underset{n\rightarrow \infty}{\sim} dn^{\beta+1}, \beta\in (0,1]$ and $\mu(n)\underset{n\rightarrow \infty}{\sim} \frac{b}{n^{\alpha+1}}$ with $\alpha \in (0,1)$ and $b>0$, then 
\begin{enumerate}
\item $\theta=\infty$ for $\beta<1-\alpha$, 
\item  $\theta=0$ for  $\beta>1-\alpha$,
\item $\theta=\frac{b}{d}\frac{1}{\alpha(1-\alpha)}\in (0,\infty)$ for $\beta=1-\alpha$.
\end{enumerate}  
\end{proposition}

According to Theorem \ref{cditheorem}, if $b/d<\alpha(1-\alpha)$, i.e $\theta<1$,  then the process comes down from infinity. We shall see later that it does not always entail the non-explosion of the process $(\#\Pi(t),t\geq 0)$, see the last statement in Theorem \ref{stablefragtheorem}.

In the next section, we briefly summarize some results on the explosion of pure branching processes.
\subsection{Explosion in branching processes}\label{explosionbranchingsec}\

Consider an immortal pure branching process $(N_t,t\geq 0)$ with offspring measure $\mu$, namely a process whose generator is $\mathcal{L}^{f}$. It is well-known that some of these processes can explode in finite time even though the generator is conservative, i.e $\mu(\infty)=0$. If one denotes by $\varphi$ the generating function of the renormalized measure $\mu(\cdot)/\mu(\mathbb{N})$, then explosion occurs if and only if \[\int^{1}\frac{\ddr x}{x-\varphi(x)}<\infty \quad (\text{Dynkin's condition}).\]
We refer the reader to Harris' book \cite[Chapter V, Section 9, Theorem 9.1]{Harris}.
We now recall a necessary and sufficient condition due to Doney \cite{Doney1984} and Schuh \cite{Schuh82}.
for any $n\geq 1$,
\begin{equation}\label{elldef} \ell(n):=\sum_{k=1}^{n}\bar{\mu}(k),\end{equation}
where for any $k\in \mathbb{N}$, $\bar{\mu}(k):=\mu(\{k,k+1,...\})$. The process $(N_t,t\geq 0)$ whose generator is $\mathcal{L}^f$, see \eqref{generator}, explodes if and only if
\begin{equation}\label{explosionbranching} \sum_{n=1}^{\infty}\frac{1}{n\ell(n)}<\infty \quad (\text{Doney's condition}).\end{equation}
We mention that Doney's condition cannot be simplified, as Grey \cite{Grey1989} has shown that explosion of a pure branching process cannot be expressed in terms of a moment condition on $\mu$. A simple application of Fubini's theorem shows that
\begin{equation}\label{ellformula}
\ell(n)=\sum_{k=1}^{\infty} (k \wedge n) \mu(k)=\sum_{k=1}^{n}k\mu(k)+n\bar{\mu}(n+1).
\end{equation}

Doney's condition does not require to work with the generating function of $\mu$ and can be used for building many examples of explosive branching processes.
If for instance, $\ell(n)\geq (\log n)^{r}$ for large enough $n$, with $r>1$, then \eqref{explosionbranching} is satisfied and the branching process $(N_t,t\geq 0)$ explodes almost-surely. Similarly if $\ell(n)\leq (\log n)^{r}$ for large enough $n$ with $r\leq 1$, then  \eqref{explosionbranching} is not satisfied and the process does not explode.

We now introduce a condition on the map $\ell$.

\noindent Condition $\mathbb{H}$: there exists an eventually non-decreasing function $g$ such that $\int^{\infty}\frac{\ddr x}{xg(x)}<\infty$ and
\begin{equation*}
\ell(n)\geq g(\log n) \log n \text{ for large enough } n. \qquad (\mathbb{H})
\end{equation*}
This latter condition covers a rather broad class of splitting measures since for instance, all measures $\mu$ for which, for large enough $n$  \[\ell(n)\geq (\log^{k} n)^{r}\log^{k-1}n\times\cdots \times \log^{2} n \log n,\] with $k\geq 1$ and $r>1$ (where $\log^{k}$ denotes the $k$-iterated logarithm), satisfy $\mathbb{H}$. We also stress that if $\mu$ satisfies \eqref{regularmu} with $\alpha\in (0,1)$ then
\begin{center}
$\ell(n)\underset{n\rightarrow \infty}{\sim} \frac{b}{\alpha(1-\alpha)}n^{1-\alpha}$ for some $b>0$.
\end{center}
In particular, if $\mu$ satisfies \eqref{regularmu}, then condition $\mathbb{H}$ is satisfied.

A simple comparison of the series $\sum_{n\geq 1}\frac{1}{n\ell(n)}$ with the integral $\int^{\infty}\frac{\ddr x}{xg(x)}$ shows that if $\mathbb{H}$ holds then Doney's condition for explosion \eqref{explosionbranching} is satisfied. Condition $\mathbb{H}$ will enable us 
to find rather sharp conditions for explosion when coalescences are taken into account.

\begin{remark}\label{rem}
\begin{enumerate}
\item Doney's proof is based on the representation of the explosion time of a branching process as a perpetual integral for a continuous-time left-continuous random walk. Such approach for studying explosion is classical for processes that are obtained through random time changes of  other processes. We refer for instance to the recent works of D\"oring and Kyprianou \cite{DoeringKyprianou}, K\"uhn \cite{Kuhn} and Li and Zhou \cite{2018arXiv180905759L}.
We shall not follow this approach here but will look for sufficient conditions based on ``local" estimates on the generator. We will establish in the forthcoming Section \ref{suffcondexplctmc}, a new sufficient condition for explosion of continuous-time Markov chains, see Theorem \ref{CTMCtheorem}.  This condition can be seen as belonging to the methods of Lyapunov functions, see e.g. Chow and Khasminskii \cite{Chow} and Menschikov and Petretis \cite{menshikov} for recent works on this approach.
\item 
We are not aware of any proof of Doney's result based on Lyapunov functions.   However, by setting $\ell(x)=\ell(\lfloor x\rfloor)$ for all $x\geq 1$,  we observe that if $x\mapsto \ell(e^{x}) /x$ is non-decreasing, then $\ell$ satisfies $\mathbb{H}$ with $g(x):=\ell(e^{x})/x$ as soon as Doney's condition holds : $\sum_{n\geq 1}\frac{1}{n \ell(n)}<\infty$. There are nevertheless examples of measures $\mu$, satisfying Doney's condition \eqref{explosionbranching} for which this function $g$ is not monotone.
\end{enumerate}
\end{remark}


\section{Main results}\label{results}
Consider a simple EFC process $(\Pi(t),t\geq 0)$  with coalescence measure $\Lambda$ and splitting measure $\mu$. We recall that the boundary $\infty$ is said to be an exit (respectively, an entrance), if the process $(\#\Pi(t),t\geq 0)$ can reach $\infty$ but can not leave from $\infty$ (respectively, can leave from $\infty$ but can not reach $\infty$).  The boundary $\infty$ is called regular if the process  $(\#\Pi(t),t\geq 0)$ can both enter $\infty$ and leave $\infty$. By convention, when we say that $\infty$ is an entrance (respectively an exit), it is always implicitely assumed that $\#\Pi(0)=\infty$ (respectively $\#\Pi(0)<\infty$). 

We first provide some general conditions on the coalescence measure $\Lambda$ and the splitting measure $\mu$ ensuring that the boundary $\infty$ is either an exit or an entrance.

\begin{theorem}[Explosion and exit]\label{suffcondpropexit}  If $n\mapsto \ell(n)$ satisfies condition $\mathbb{H}$ and \[\underset{n\rightarrow \infty}{\lim}\  \frac{\Phi(n)}{n\ell(n)}=0,\] then 
$\infty$ is an exit boundary.
\end{theorem}
\begin{remark}
Under the assumption of Theorem \ref{suffcondpropexit}, the process $(\Pi(t),t\geq 0)$ reaches in a finite time almost-surely, a proper partition with infinitely many blocks. 
\end{remark}
A direct application of Theorem \ref{suffcondpropexit} yields the following examples.
\begin{example}\label{exampleexit} Let $d>0$ and $b>0$.\
\begin{enumerate}
\item Assume that $\Phi(n)\underset{n\rightarrow \infty}{\sim} d n^{1+\beta}$ with $\beta\in (0,1)$
and $\ell(n)\geq b n(\log n)^{r}$ for large enough $n$ with $r\in \mathbb{R}$. Then $\infty$ is an exit boundary.
\item Assume that $\Phi(n)\underset{n\rightarrow \infty}{\sim} d n (\log n)^{\beta}$ with $\beta>0$ and $\ell(n)\geq b(\log n)^{r}$ for large enough $n$ with $r>\beta$. Then $\infty$ is an exit boundary.
\end{enumerate}
\end{example}
\begin{theorem}[Non-explosion and entrance]\label{suffcondpropentrance}   If \[\sum_{n=2}^{\infty}\frac{n}{\Phi(n)}\bar{\mu}(n)<\infty,\] 
	then $(\#\Pi(t),t\geq 0)$ does not explode almost-surely.

If furthermore, $\sum_{n\geq 2}\frac{1}{\Phi(n)}<\infty$, then
 $\infty$ is an entrance boundary.
\end{theorem}

A direct application of Theorem \ref{suffcondpropentrance} yields:
\begin{example}\label{exampleentrance} Let $d>0$.\
\begin{enumerate}
\item Assume that $\Phi(n)\underset{n\rightarrow \infty}{\sim} dn^{1+\beta}$ with $\beta\in (0,1]$. If $\sum_{n\geq 1}\frac{\bar{\mu}(n)}{n^{\beta}}<\infty$, then the process does not explode and $\infty$ is an entrance boundary.
\vspace*{1mm}
\item Assume that $\Phi(n)\underset{n\rightarrow \infty}{\sim} d n (\log n)^{\beta}$ with $\beta>0$. If $\sum_{n\geq 1}\frac{\bar{\mu}(n)}{(\log n)^{\beta}}<\infty$, then the process does not explode. If furthermore $\beta>1$, then $\infty$ is an entrance boundary.

\end{enumerate}
\end{example}
\begin{remark} When the coalescences are driven by a pure Kingman coalescent, $\Lambda=c_{\mathrm{k}}\delta_0$ with $c_{\mathrm{k}}>0$, $\Phi(n)=c_{\mathrm{k}}\binom{n}{2}\underset{n\rightarrow \infty}{\sim} \frac{c_{\mathrm{k}}}{2}n^{2}$ and by Example \ref{exampleentrance}-(1), with $\beta=1$, we see that if $\sum_{n\geq 1}\frac{\bar{\mu}(n)}{n}<\infty$, then $\infty$ is an entrance boundary. It agrees with a log-moment assumption on $\mu$ such that $\sum_{k\geq 1}\mu(k) \log k <\infty$  and we recover with a different method a result of Lambert, see \cite[Theorem 2.3]{Lambert:2005bq}.
\end{remark}
Our next results treat simple EFC processes with regularly varying coagulation/splitting measures. Examples of simple EFC processes for which the boundary is regular are exhibited. 

\begin{theorem}\label{stablefragtheorem} Assume that $\Phi(n)\underset{n\rightarrow \infty}{\sim} dn^{1+\beta}$ with $d>0$ and $\beta\in (0,1)$ and $\mu(n)\underset{n\rightarrow \infty}{\sim} \frac{b}{n^{\alpha+1}}$ with $b>0$ and $\alpha\in (0,\infty)$. Then $n\ell(n)\underset{n\rightarrow \infty}{\sim}\frac{b}{\alpha(1-\alpha)}n^{2-\alpha}$ and
\begin{itemize}
\item if $\alpha+\beta<1$, then 
$\infty$ is an exit boundary,
\vspace*{1mm}
\item if $\alpha+\beta>1$, then 
 $\infty$ is an entrance boundary,
\vspace*{1mm}
\item if $\alpha+\beta=1$ and further,
\vspace*{2mm}
\begin{itemize}
\item if $b/d>\alpha(1-\alpha)$, then
$\infty$ is an exit boundary,
\vspace*{2mm}
\item if $\frac{\alpha \sin(\pi \alpha)}{\pi}<b/d<\alpha(1-\alpha)$, then
$\infty$ is a regular boundary,
\vspace*{2mm}
\item if $b/d<\frac{\alpha \sin(\pi \alpha)}{\pi}$, then 
$\infty$ is an entrance boundary.
\end{itemize}
\end{itemize}
\end{theorem}
\begin{remark} The first two statements of Theorem \ref{stablefragtheorem} are  consequences of    Theorem \ref{suffcondpropexit} and Theorem \ref{suffcondpropentrance}, respectively. The third statement is shown in Lemma \ref{part2proofregularvarytheorem}. 
\end{remark}

\noindent The figure below represents the different possible regimes for the boundary $\infty$, found in Theorem \ref{stablefragtheorem} when $\alpha+\beta=1$, according to the location of ratio $b/d$.
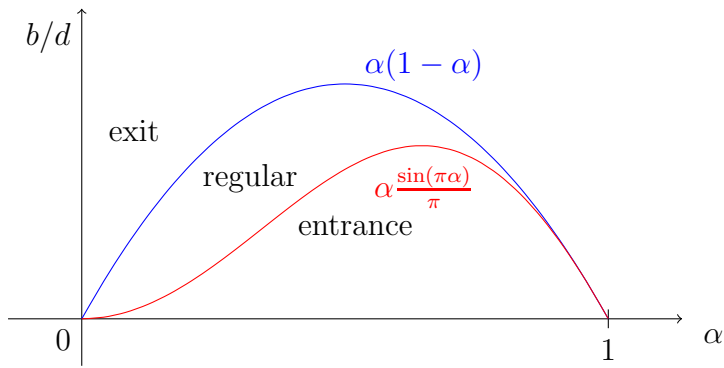
\begin{figure}[h!]
\centering 

\begin{tikzpicture}[xscale = 7,yscale = 12.5]
  \draw [->] (0,-0.05) -- (0,0.33);
  \draw [->] (-0.14,0) -- (1.14,0);
  \draw (0,0) node[below left] {0};
  \draw (1,0.01)--(1,-0.01) node[below] {1};
  \draw (1.2,0) node[below] {$\alpha$};
  \draw (0,0.3) node[left] {$b/d$};
  \draw [domain=0:1, smooth, variable=\x, blue] plot ({\x}, {\x*(1-\x)});
  \draw [color = blue] (0.65,0.27) node  {$\alpha(1-\alpha)$};
  \draw [domain=0:1, smooth, variable=\x, red] plot ({\x}, {\x*sin(3.14159*deg(\x))/3.14159});
  \draw [color = red] (0.65,0.14) node {$\alpha\frac{\sin(\pi \alpha)}{\pi}$};
  \draw (0.1,0.2) node {exit};
  \draw (0.32,0.15) node {regular};
  \draw (0.52,0.1) node {entrance};
\end{tikzpicture}
\caption{Boundary classification when $\Phi(n)\underset{n\rightarrow \infty}{\sim} dn^{2-\alpha}$ and $\mu(n)\underset{n\rightarrow \infty}{\sim}\frac{b}{n^{1+\alpha}}$.}
\label{graph}
\end{figure}

The next proposition describes more precisely the behavior of the process $(\Pi(t),t\geq 0)$, with regularly varying coalescence-splitting measures, when the block-counting process has  $\infty$ as regular boundary.
We establish that the boundary $\infty$ is regular for itself, in the sense that the block-counting process returns to $\infty$ immediately after having left it.
\begin{proposition}\label{regularforitself} Suppose that the assumptions of Theorem \ref{stablefragtheorem} hold. If $\beta=1-\alpha$ and $\frac{\alpha \sin(\pi \alpha)}{\pi}<b/d<\alpha(1-\alpha)$, then the process $(\Pi(t),t\geq 0)$ started from a partition with infinitely many blocks comes down from infinity and returns instantaneously to a proper partition with infinitely many blocks almost-surely.
\end{proposition}

The critical cases in the last statement of Theorem \ref{stablefragtheorem} for which the ratio $\frac{b}{d}$ equals $\frac{\alpha\sin(\pi \alpha)}{\pi}$ or $\alpha(1-\alpha)$ seem to be requiring finer arguments. We find in the next proposition a class of coalescence and splitting measures for which the critical value $\frac{b}{d}=\frac{\alpha\sin(\pi \alpha)}{\pi}$ can be handled.
\begin{proposition}\label{criticalvalue} Let $b,d>0$, $\alpha\in (0,1)$ and $h$ be a measurable function on $[0,1]$ such that $h\geq 1$. Set $\beta=1-\alpha$. Assume that $\Lambda(\ddr x)= d\frac{\beta(\beta+1)}{\Gamma(1-\beta)}x^{-\beta}h(x)\ddr x$ and that $\mu(n)=\frac{b}{n^{1+\alpha}}$ for all integer $n\geq 1$. Then, when $\frac{b}{d}=\frac{\alpha\sin(\pi \alpha)}{\pi}$, the boundary $\infty$ is an entrance.
\end{proposition}
We study now the slower regime of coalescences for which for some $\beta>0$ \begin{equation}\label{slowvarphi}\Phi(n)\underset{n\rightarrow \infty}{\sim} d n(\log n)^{\beta}. 
\end{equation}
We refer the reader to \cite[Section 2.2]{cdiEFC} for conditions on the coalescence measure $\Lambda$ entailing that the function $\Phi$ has these asymptotics. We shall see that when \eqref{slowvarphi} holds, there is no regular regime when $n\ell(n)$ is of the same order as $\Phi(n)$.
\begin{theorem}\label{logcoal} Let $\beta>0$, $d>0$ and $b>0$ and $\alpha>0$. Assume that $\Phi(n)\underset{n\rightarrow \infty}{\sim} d n(\log n)^{\beta}$ and $\mu(n)\underset{n\rightarrow \infty}{\sim} b \frac{(\log n)^{\alpha}}{n^2}$. Then $n\ell(n)\underset{n\rightarrow \infty}{\sim} \frac{b}{1+\alpha}n(\log n)^{\alpha+1}$ and
\begin{itemize}
\item if $\beta<1+\alpha$, then 
$\infty$ is an exit boundary,
\vspace*{1mm}
\item if $\beta>1+\alpha$, then 
$\infty$ is an entrance boundary,
\vspace*{1mm}
\item if $\beta=1+\alpha$, then
\vspace*{1mm}
\begin{itemize}
\item if $b/d>1+\alpha$, 
then $\infty$ is an exit boundary,
\vspace*{1mm}
\item if $b/d<1+\alpha$, $(\#\Pi(t),t\geq 0)$ 
then $\infty$ is an entrance boundary.
\end{itemize}
\end{itemize}
\end{theorem}

\begin{remark} 
The first two statements of Theorem \ref{logcoal} are  consequences of  Theorem \ref{suffcondpropexit} and Theorem \ref{suffcondpropentrance}, respectively. The proof 
of the third statement is deferred to Lemma \ref{thirdpartTheoremslow}. 
The critical case $b/d=1+\alpha$ in the last statement of Theorem \ref{logcoal} remain unsolved. 
\end{remark}

\section{Explosion of a general CTMC on $\mathbb{N}$}\label{suffcondexplctmc}
We state in this section sufficient conditions for explosion and non-explosion of a general continuous-time Markov chain taking values in $\mathbb{N}$. Explosion in this setting corresponds to accumulations of large jumps in compact intervals of time that are pushing the process to $\infty$.  We believe the results of Section \ref{sec:explosion} of independent interest.

Consider an infinitesimal generator $\mathscr{L}=\mathscr{L}^{-}+\mathscr{L}^{+}$ acting on all bounded function $g:\mathbb{N}\rightarrow \mathbb{R}_+$ and all $n\in \mathbb{N}$ as follows
\begin{align}\label{generalgenerator}
\mathscr{L}^{-}g(n)&=\sum_{k=1}^{n-1}\big(g(n-k)-g(n)\big)p^{-}_{n,k} \text{ and }
\mathscr{L}^{+}g(n)=\sum_{k=1}^{\infty}\big(g(n+k)-g(n)\big)p^{+}_{n,k}
\end{align}
where $p^{+}_{n,k}\in [0,\infty)$ and $p^{-}_{n,k}\in [0,\infty)$ are respectively the rates of positive and negative jumps.

Standard theory of Markov processes, see e.g \cite{Anderson}, ensures that there exists a unique continuous-time Markov chain $(N_t,t\geq 0)$, taking values in $\mathbb{N}\cup \{\infty\}$, with generator $\mathscr{L}$, absorbed at $\infty$ (viewed as a cemetery point) after explosion. Denote the first explosion time by $\tau^{+}_\infty:=\inf\{t>0; N_{t-}=\infty\}$.

\subsection{Explosion}\label{sec:explosion}
The following theorem provides a sufficient condition for explosion to occur with positive probability.  Theorem \ref{CTMCtheorem} has several precursors in the literature for positive real-valued Markov processes without negative jumps, see \cite{Lietal} and the references therein.

For any $a>0$ and for any $n \in \mathbb{N}^{\star}$, set $g_a(n):=n^{1-a}$ and $G_a(n):=-\frac{1}{n^{1-a}}\mathscr{L}g_a(n)$.
\begin{theorem}\label{CTMCtheorem} If there  exist $a>1$ and an eventually non-decreasing positive function $g$ satisfying $\int^{\infty}\frac{\ddr x}{xg(x)}<\infty$ such that  for all large enough $n$
\begin{equation}\label{generalcondition} G_a(n)\geq g(\log n)\log n,
\end{equation}
 then $\mathbb{P}_n(\tau_\infty^{+}<\infty)>0$ for all large enough $n\in \mathbb{N}$. 

If moreover, the process is irreducible in $\mathbb{N}$, then   $\mathbb{P}_n(\tau_\infty^{+}<\infty)>0$ for all $n\in \mathbb{N}$.
\end{theorem}
%
%
We adapt the method of Li et al. in  \cite[Section 5]{Lietal} where the explosion of nonlinear branching processes is studied. We stress that in our framework, the process has both positive and negative jumps, moreover large negative jumps may occur along large coalescence events. We establish Theorem \ref{CTMCtheorem} with the help of several lemmas. 
%
%

The proof of Theorem \ref{CTMCtheorem} relies on some estimates on exit probabilities from an interval $[|n,m|]$ for the process $(N_t,t\geq 0)$. The latter will be obtained using the following martingale.

\begin{lemma}[Martingale]\label{martingale} Let $n_0$ be a fixed integer and let $n$ and $m$ be two integers such that $n>n_0>m$. Set $T:=\tau^{-}_{n}\wedge \tau^{+}_{m}$.
For $a>1$,  the process \[\left(N_{t\wedge T}^{1-a}\exp \left(\int_0^{t\wedge T}G_a(N_s)\ddr s\right), t\geq 0\right)\] is a bounded $(\mathcal{F}_t)$-martingale, and
\begin{equation}\label{martingale_lem}
\mathbb{E}_{n_0}\left[N_{T}^{1-a}\exp \left(\int_0^{T}G_a(N_s)\ddr s\right)\right]\leq n_0^{1-a}.
\end{equation}
\end{lemma}

\begin{proof} 
Recall $g_a(n):=n^{1-a}$ and $G_a(n):=-\frac{1}{n^{1-a}}\mathscr{L}g_a(n)$ for all $n\geq 1$. By Dynkin's formula, $N_{t\wedge T}^{1-a}- \int_0^{t\wedge T}\mathscr{L}g_a(N_s)\ddr s$ is a local martingale.
Since the quadratic variation process vanishes, i.e $\langle N_{\cdot\wedge T}^{1-a},  \int_0^{\cdot\wedge T}G_a(N_s)\ddr s\rangle=0$, by the product rule of It\^o's formula, we have that $\left(N_{t\wedge T}^{1-a}\exp \left(\int_0^{t\wedge T}G_a(N_s)\ddr s\right), t\geq 0\right)$ is a local martingale. Observe that for each $t>0$, $\int_0^{t\wedge T}G_a(N_s)\ddr s$ is  bounded.
Since $a>1$, $N_{t\wedge T}^{1-a}$ is bounded from above uniformly for all $t>0$, and by \cite[Theorem I.51]{Protter}, the process $\left(N_{t\wedge T}^{1-a}\exp \left(\int_0^{t\wedge T}G_a(N_s)\ddr s\right), t\geq 0\right)$ is a  martingale.
The inequality (\ref{martingale_lem})  follows from Fatou's lemma.\qed

\end{proof}

\begin{lemma}[Estimates on exit probabilities]\label{exitestimates} Under the assumption of Theorem \ref{CTMCtheorem}. Recall $a>1$. For $n$ large enough and any $m>n_0>n$, we have
	 \begin{equation}
\label{0.1} \mathbb{P}_{n_0}(\tau_n^{-}<\tau_m^{+})\leq \left(\frac{n}{n_0} \right)^{a-1}\!,
\end{equation}
\begin{equation}\label{0.2} \mathbb{P}_{n_0}(\tau_n^{-}=\tau_m^{+}=\infty)=0.
\end{equation}
For any $u>0$, we have
\begin{equation}\label{0.3} \mathbb{P}_{n_0}(\tau_n^{-}>\tau_m^{+}>u)\leq \left(\frac{n_0}{m}\right)^{1-a}e^{-ug(\log n) \log n}.
\end{equation}
For $0<\delta<1/(2a-1)$, set $n=[n_0^{1-\delta}]+1$ and $m=[n_{0}^{1+\delta}]$ and define $t(y):=\frac{1}{g(\log y^{1-\delta}) }$ for any $y\geq 1$. We have  \begin{equation}\label{0.4}
\mathbb{P}_{n_0}(\tau_m^{+}>t(n_0))\leq (1+2^{a-1}) n_0^{\delta(1-a)}.
\end{equation}
\end{lemma}
\begin{proof}
By assumption, $G_a(n)\geq g(\log n)\log n$ for large enough $n$. In particular, for any $s<\tau_{n}^{-}$, $N_s\geq n$ and if $n$ is large enough then $G_a(N_s)\geq 0$. By Lemma \ref{martingale},
\begin{align*}
n_0^{1-a}&\geq \mathbb{E}_{n_0}\left(N_{T}^{1-a}\exp \left(\int_{0}^{T}G_a(N_s)\ddr s\right)\mathbbm{1}_{\{\tau_n^{-}<\tau_m^{+}\}}\right)\\
&\geq n^{1-a}\mathbb{P}_{n_0}(\tau_n^{-}<\tau_m^{+}).
\end{align*}
Thus, \eqref{0.1} is established.
Applying Lemma \ref{martingale} to the bounded stopping time $T\wedge u$. We get
\begin{equation}\label{0.5}
n_0^{1-a}\geq \mathbb{E}_{n_0}\left(N_{u}^{1-a}\exp \left(\int_{0}^{u}G_a(N_s)\ddr s\right)\mathbbm{1}_{\{\tau_n^{-}\wedge \tau_m^{+}>u\}}\right).
\end{equation}
Moreover, when $s\leq u< \tau^{-}_{n}\wedge \tau_{m}^{+}$, $N_{s}\in [|n,m|]$ and thus $G_{a}(N_s)\geq g(\log n)\log n$. The inequality \eqref{0.5} provides
$$\mathbb{P}_{n_0}(\tau_{n}^{-}\wedge \tau_{m}^{+}>u)\leq e^{-u g(\log n)\log n}\left(\frac{n_0}{m}\right)^{1-a}.$$
Letting $u$ go to $\infty$, yields \eqref{0.2}.
Let $u>0$. Similarly, by Lemma \ref{martingale},
\begin{align*}
n_0^{1-a}&\geq \mathbb{E}_{n_0}\left(N_{u}^{1-a}\exp \left(\int_{0}^{u}G_a(N_s)\ddr s\right)\mathbbm{1}_{\{\tau_n^{-}\wedge \tau_m^{+}>u\}}\right)\\
&\geq m^{1-a}\mathbb{E}_{n_0}\left(e^{u g(\log n)\log n}\mathbbm{1}_{\{u<\tau_m^{+}<\tau_n^{-}\}}\right),
\end{align*}
this provides \eqref{0.3}.
Let $\delta\in (0,1)$, and set $n=[n_0^{1-\delta}]+1$ and $m=[n_{0}^{1+\delta}]$ and define
$t(y)= 1/(g(\log y^{1-\delta}))$ for any $y\geq 1$. We show now that
$\mathbb{P}_{n_0}(\tau_n^{-}>\tau^{+}_{m}>t(n_0))\leq n_0^{a\delta-1}$.
By \eqref{0.3}, \begin{align*}
\mathbb{P}_{n_0}(\tau_n^{-}>\tau_m^{+}>t(n_0))&\leq \left(\frac{n_0}{m}\right)^{1-a}e^{-t(n_0) g(\log n)\log n}\\
&= \left(\frac{n_0}{m}\right)^{1-a}e^{-(g(\log n_0^{1-\delta}))^{-1} g(\log n)\log n}.
\end{align*}
Since $n=[n_0^{1-\delta}]+1$, we have that $\log n\geq \log n_0^{1-\delta}$ and
$(g(\log n_0^{1-\delta}))^{-1} g(\log n)\log n\geq \log \left(n_0^{1-\delta}\right)$. Recall $a>1$ and $m=[n_{0}^{1+\delta}]\leq n_0^{1+\delta}$. This entails
\begin{equation}\label{estimate_exit}
\begin{split}
\mathbb{P}_{n_0}(\tau_n^{-}>\tau_m^{+}>t(n_0))
&\leq \left(\frac{n_0}{m}\right)^{1-a}e^{-\log \left(n_0^{1-\delta}\right)}\\
&= \left(\frac{n_0}{m}\right)^{1-a}n_0^{\delta-1}\leq n_0^{(a-1)\delta}n_0^{\delta-1}=n_0^{a\delta -1}.
\end{split}
\end{equation}
One has
\[\mathbb{P}_{n_0}(\tau_m^{+}>t(n_0))\leq \mathbb{P}_{n_0}(\tau_n^{-}>\tau_{m}^{+}>t(n_0))+\mathbb{P}(\tau_{n}^{-}<\tau_{m}^{+})+\mathbb{P}(\tau_n^{-}=\tau_{m}^+=\infty).\]
By \eqref{0.1}, we finally get
\begin{equation}\label{firstpart}\mathbb{P}_{n_0}(\tau_m^{+}>t(n_0))\leq n_0^{a \delta -1}+\left(\frac{n}{n_0}\right)^{a-1}.
\end{equation}
Since $n=[n_0^{1-\delta}]+1$, then \[\left(\frac{n}{n_0}\right)^{a-1}\leq \left(\frac{n_0^{1-\delta}+1}{n_0}\right)^{a-1}=n_0^{-\delta(a-1)}\left(1+n_0^{\delta-1}\right)^{{a-1}}.\]
Since $a>1$ and $\delta<\frac{1}{2a-1}<1$,  $n_0^{\delta-1}\leq 1$ and we get $\left(\frac{n}{n_0}\right)^{a-1}\leq 2^{a-1}n_0^{-\delta(a-1)}$. Moreover, $\delta<\frac{1}{2a-1}$, thus $a\delta-1\leq \delta(1-a)$ and
we deduce from \eqref{firstpart},
\[\mathbb{P}_{n_0}(\tau_m^{+}>t(n_0))\leq n_0^{a \delta -1}+2^{a-1}n_{0}^{-\delta(a-1)}\leq (2^{a-1}+1)n_0^{\delta(1-a)}=(1+2^{a-1})n_0^{\delta(1-a)}.\]
\qed
\end{proof}
Using the estimates on the exit probabilities given in Lemma \ref{exitestimates}, we will show that there is an accumulation of positive jumps pushing up the process to infinity in finite time with positive probability. This will finish the proof of Theorem \ref{CTMCtheorem}.
\begin{lemma}\label{keylemma} Under the assumptions of Theorem \ref{CTMCtheorem}, for any $a>1$ and $0<\delta<\frac{1}{2a-1}$, we have for large enough $n_0$, 
\begin{equation}\label{lowerboundexplosion}\mathbb{P}_{n_0}\left(\tau_\infty^{+}\leq \int_{\frac{1}{1+\delta}\log n_0^{1-\delta}}^{\infty}\frac{\ddr v}{vg(v)}\right)\geq \prod_{k=0}^{\infty}h_a(k,n_0)>0
\end{equation}
with $h_a(k,n_0):=1-(2^{a-1}+1)\left(\frac{1}{n_0^{\delta(a-1)}}\right)^{(1+\delta)^{k}}$  for all $k\in \mathbb{Z}_+$.
\end{lemma}
\begin{proof}
Define  $(\tilde{\tau}_k)$ recursively by $\tilde{\tau}_0:=0$ and \[\tilde{\tau}_{k+1}:=\tau^{+}_{N_{\tilde{\tau}_k}^{1+\delta}}\circ \theta_{\tilde{\tau}_k},\]
where $\theta_t$ is the shift operator and with the convention that on the event $\tilde{\tau}_{k}=\infty$ we set $N_{\tilde{\tau}_k}=1$. By construction, under $\mathbb{P}_{n_0}$ for any $k\geq 0$,
$N_{\tilde{\tau}_k}\geq n_0^{(1+\delta)^{k}}$ a.s. on the event $\{\tilde{\tau}_k<\infty\}$.
Therefore, if $\underset{k\rightarrow \infty}{\lim} \tilde{\tau}_k=:\tilde{\tau}_{\infty}<\infty$, \[N_{\tilde{\tau}_{\infty}-}= \underset{k\rightarrow \infty}{\lim} N_{\tilde{\tau}_k}\geq \underset{k\rightarrow\infty}{\lim} n_0^{(1+\delta)^{k}}=\infty.\]
Hence, recalling that $\tau^{+}_\infty:=\inf\{t>0; N_{t-}=\infty\}$, we have $\tilde{\tau }_\infty\geq \tau^{+}_\infty$ a.s. Intuitively, the event $\{\tilde{\tau }_\infty<\infty\}$ corresponds to having an accumulation of positive jumps (in our case fragmentations events) in which the number of fragments at time $t-$ increases at time $t$ by a factor of order $N_{t-}^{\delta}$. We study $\tilde{\tau}_{\infty}$. Recall $t(y)=\frac{1}{g(\log y^{1-\delta})}$. Since the function $g$ is non-decreasing, if $\tilde{\tau}_k<\infty $, we have
\begin{align*}
 t(N_{\tilde{\tau}_k})&\leq (g( \log n_0^{(1+\delta)^{k}(1-\delta)}))^{-1}= (g\big((1+\delta)^{k}\log n_0^{1-\delta})\big))^{-1}.
\end{align*}
Therefore, by the assumption on $g$, for any $m\geq 0$
\begin{equation}\label{upperboundintegralg}
\begin{split}
 \sum_{k=m}^{\infty}t(N_{\tilde{\tau}_k})
 &\leq \sum_{k=m}^{\infty} \frac{1}{g\big((1+\delta)^{k}\log n_0^{1-\delta}\big)}\\
 &=\frac{1+\delta}{\delta}\sum_{k=m}^{\infty}\frac{(1+\delta)^k-(1+\delta)^{k-1}}{(1+\delta)^{k} g\big((1+\delta)^{k}\log n_0^{1-\delta}\big)}\\
 &\leq\int_{(1+\delta)^{m-1}}^\infty\frac{1}{ug(u \log n_0^{1-\delta})}\ddr u = \int_{(1+\delta)^{m-1}\log n_0^{1-\delta}}^{\infty}\frac{\ddr v}{vg(v)} <\infty\quad\text{ a.s.}
\end{split}
\end{equation}
Conditionally on $N_{\tilde{\tau}_k}$, one has using \eqref{0.4},
\[\mathbb{P}_{N_{\tilde{\tau}_k}}(\tau_{N_{\tilde{\tau}_k}^{1+\delta}}<t(N_{\tilde{\tau}_k})) \geq 1-(1+2^{a-1})N_{\tilde{\tau}_k}^{\delta(1-a)}\geq 1-(1+2^{a-1})\left(\frac{1}{n_{0}^{\delta(a-1)}}\right)^{\!\!\!(1+\delta)^{k}}\!\!\!\!=:h_a(k,n_0).\]
Using  the strong Markov property in the third inequality below, we get for $n_0$ large enough,
\begin{align}\label{lowerboundtildetau}
\mathbb{P}_{n_0}\left(\tilde{\tau}_\infty\leq \sum_{k=0}^{\infty}t(N_{\tilde{\tau}_k})\right) &\geq \mathbb{P}_{n_0}(\tilde{\tau}_{k+1}-\tilde{\tau}_k<t(N_{\tilde{\tau}_k}), \forall k\geq 0) \nonumber \\
&\geq  \prod_{k=0}^{\infty}h_a(k,n_0)>0.
\end{align}
Combining \eqref{upperboundintegralg} and \eqref{lowerboundtildetau} and using the fact that $\tilde{\tau}_\infty\geq \tau^{+}_\infty$ a.s, we get
\[\mathbb{P}_{n_0}\left(\tau_\infty^{+}\leq \int_{\frac{1}{1+\delta}\log n_0^{1-\delta}}^{\infty}\frac{\ddr v}{vg(v)}\right)\geq \prod_{k=0}^{\infty}\left(1-(2^{a-1}+1)\left(\frac{1}{n_0^{\delta(a-1)}}\right)^{(1+\delta)^{k}}\right)>0\]
which ensures that explosion has a positive probability when the process starts from a large enough $n_0\in \mathbb{N}$. \qed
\end{proof}
\noindent \textbf{Proof of Theorem} \ref{CTMCtheorem}. This is a consequence of Lemma \ref{keylemma}. \qed

In order to see how the positive and negative jumps interplay in the condition \eqref{generalcondition}, we introduce the following functions. For any $a>0$,
\begin{equation} G_a^{-}(n):=-\frac{1}{n^{1-a}}\mathscr{L}^{-}g_a(n), \quad G^{+}_a(n):=-\frac{1}{n^{1-a}}\mathscr{L}^{+}g_a(n).
\end{equation}
Note that for any $a$, since the sum in the expression of $\mathscr{L}^{-}$ is finite, then $G^{-}_a$ is well-defined. When $a\geq 1$, since $g_a$ is bounded, the infinite sum in the expression of $G^{+}_a$ is meaningful.
Moreover, when $a>1$, $G_a^{+}(n)\geq 0$ and $G_a^{-}(n)\leq 0$.

The next corollary states a simpler condition for explosion in terms of $G_a^{+}$ and $G_a^{-}$.
\begin{corollary}\label{genexplosionmaincor}\
Assume that there are $a>1$ and a non-decreasing  positive function $g$ such that $\int^{\infty}\frac{\ddr x}{xg(x)}<\infty$, and for large enough $n$, $G_a^{+}(n)\geq g(\log n )\log n$.

If $\gamma_a:=\underset{n\rightarrow \infty}{\limsup}\frac{-G_a^{-}(n)}{G_a^{+}(n)}<1,$
then Condition \eqref{generalcondition} holds and $\infty$ is accessible from any large enough initial state. If moreover, the process is irreducible in $\mathbb{N}$ then   $\mathbb{P}_n(\tau_\infty^{+}<\infty)>0$ for all $n\in \mathbb{N}$.
\end{corollary}
\begin{proof} The proof is straightforward as by definition of $G_a^{-}$ and $G_a^{+}$, we have for large enough $n$
\[-G_a(n)=-G_a^{-}(n)-G_a^{+}(n)=-G_a^{+}(n)\left(1-\frac{-G_a^{-}(n)}{G_a^{+}(n)}\right)\leq -cg(\log n) \log n\]
where $c:=1-\gamma_a\in(0,\infty)$, and Condition \eqref{generalcondition} holds with the function $x\mapsto cg(x)$.
\qed
\end{proof}

\subsection{Non-explosion}\label{lyapunovcond}
We recall here a classical Foster-Lyapunov's sufficient condition for non-explosion. We refer for instance to Chow and Khasminskii \cite{Chow} and the references therein.  If $f:n\mapsto f(n)$ is non-decreasing, $f(n)\underset{n\rightarrow \infty}{\longrightarrow} \infty$ and
\begin{equation}\label{lyapunov}
\mathscr{L}f(n)\leq c f(n) \text{ for all } n\geq 1
\end{equation}
for some $c>0$, then the (minimal) continuous-time Markov chain $(N_t,t\geq 0)$ with generator $\mathscr{L}$ does not explode from all initial states. 

Consider as Lyapunov function $f$ the function $g_a:n\mapsto n^{1-a}$ with $0<a<1$. For any $0<a<1$, $g_a(n)\underset{n\rightarrow \infty}{\longrightarrow} \infty$ and a simple sufficient condition entailing \eqref{lyapunov} and thus non-explosion is $\mathscr{L}g_a(n)\leq 0$ for large enough $n$.
Recall $\mathscr{L}^{\pm}g_a(n)=:-n^{1-a}G^{\pm}_a(n)$. Note that when $a<1$, $G_a^{+}(n)\leq 0$ and $G_a^{-}(n)\geq 0$ for all $n$.  One has the following proposition.
\begin{proposition}\label{genmaincornonexplosion}\
If there exists $a<1$ such that $\mathscr{L}^{+}g_a(n)<\infty$ for all $n\in \mathbb{N}$, and  \begin{equation}\label{limsupnonexp}\underset{n\rightarrow \infty}{\limsup}\frac{-G_a^{+}(n)}{G_a^{-}(n)}<1,
\end{equation}
then $\mathscr{L}g_a(n)\leq 0$ for large enough $n$ and Condition \eqref{lyapunov} holds with $f=g_a$.
\end{proposition}
\begin{proof}
For any $n\in \mathbb{N}$, $\mathscr{L}g_a(n)=-n^{1-a}G_a^{-}(n)\left(1- \frac{-G_a^{+}(n)}{G_a^{-}(n)}\right)$. If \eqref{limsupnonexp} holds, then there is a large $n_0$ such that for all $n\geq n_0$, $\frac{-G_a^{+}(n)}{G_a^{-}(n)}<1$. Since for all $n\geq n_0$, $-n^{1-a}G_a^{-}(n)\leq 0$ and
$1-\frac{-G_a^{+}(n)}{G_a^{-}(n)}\geq 0$, we have that $\mathscr{L}g_a(n)\leq 0$ for all $n\geq n_0$. Set $c:=\underset{1\leq n\leq n_0-1}{\max}\left\lvert\frac{\mathscr{L}g_a(n)}{g_a(n)}\right \lvert$, then we have $\mathscr{L}g_a(n)\leq c g_a(n)$ for all $n\geq 1$.\qed
\end{proof}

\section{Proofs of the main results}\label{proofs}
\subsection{A sufficient condition for explosion of $(\#\Pi(t),t\geq 0)$}\label{explosec}\

We will study the explosion by applying Corollary \ref{genexplosionmaincor} and Proposition \ref{genmaincornonexplosion}. Recall the functions $\ell$ and $\Phi$ governing the fragmentations and the coagulations defined respectively in \eqref{elldef} and \eqref{phi}.  Theorem \ref{suffcondpropexit} is a simple consequence of the following lemma. Recall the condition $\mathbb{H}$.
\begin{lemma}\label{lemmaexplosion} Assume that $n\mapsto \ell(n)$ satisfies condition $\mathbb{H}$. If $\rho:=\underset{n\rightarrow \infty}{\limsup}\frac{\Phi(n)}{n\ell(n)}<\frac{1}{2}$ then the process $(\#\Pi(t),t\geq 0)$ explodes almost-surely. If $\rho\leq \frac{1}{4}$ then $\infty$ is an exit boundary.
\end{lemma}

The proof of Lemma \ref{lemmaexplosion} is deferred. 

\noindent \textbf{Proof of Theorem \ref{suffcondpropexit}}. Under the assumption of Theorem \ref{suffcondpropexit}, $\rho=0$ and according to Lemma \ref{lemmaexplosion} the process explodes almost-surely and has $\infty$ as an exit boundary. \qed

The condition $\rho<\frac{1}{2}$ is not a very sharp condition but it holds for quite general splitting measures. We shall see later how to improve Lemma \ref{lemmaexplosion} for coagulation and splitting measures with some regular variation properties.

We now explain the main idea of the proof of Lemma \ref{lemmaexplosion}. We shall follow a similar route in order to show Theorem \ref{stablefragtheorem} and Theorem \ref{logcoal}. Fix $p\in (0,1)$ and $n_0\in \mathbb{N}$. Recall $(\#\Pi^{(n_0)}(t),t\geq 0)$ defined in Section \ref{preliminaries}, Lemma \ref{monotonicity} and the generator $\mathcal{L}$ in \eqref{generator}.  Consider the process $(\#\Pi^{(n_0)}(t),t\geq 0)$  stopped at the first coalescence time at which the number of blocks decreases of a proportion larger than $p$.
Set \begin{equation}\label{sigmap}\sigma^{(n_0)}_p:=\inf \{t\geq 0; \#\Pi^{(n_0)}(t)\leq (1-p)\#\Pi^{(n_0)}(t-)\}.
\end{equation}
The process  $(\#\Pi^{(n_0)}(t\wedge \sigma^{(n_0)}_p),t\geq 0)$, which appears also in \cite[Section 3.2]{cdiEFC}, is a Markov process with generator $\mathcal{L}^{p}$ defined by
\[\mathcal{L}^{p}g(n)=\mathcal{L}^{c,p}g(n)+\mathcal{L}^{f}g(n)\]
with \[\mathcal{L}^{c,p}g(n):=\sum_{k=2}^{\lfloor np \rfloor }\binom{n}{k}\lambda_{n,k}\big( g(n-k+1)-g(n)\big)\]
where $\lfloor np\rfloor$ denotes the greatest integer less than or equal to $np$. 

The assumptions of Lemma \ref{lemmaexplosion} will allow us to choose $p$ small enough and $a$ close enough to $1$, in order to be able to apply Corollary \ref{genexplosionmaincor} to the minimal (unstopped) Markov process with generator $\mathcal{L}^{p}$. The latter will therefore explode with a positive probability. Estimates on the jump-time $\sigma_p^{(n_0)}$, found in \cite{coaldist11} and \cite{cdiEFC}, will ensure that the process $(\#\Pi^{(n_0)}(t\wedge \sigma_p^{(n_0)}),t\geq 0)$ explodes also with positive probability for large enough $n_0$. The fact that the process can explode with positive probability when started from any initial point $n_0\in \mathbb{N}$ is a consequence of the irreducibility of the process. Lastly, the Markov property will entail that explosion happens actually almost-surely.

Call respectively $G_a^{f}$, $G_a^{c}$ and $G_a^{c,p}$ the functions $G_a^{+}$ and $G_a^{-}$ associated to the generators $\mathcal{L}$ and $\mathcal{L}^{p}$. For any $n\geq 2$ and any $p\in (0,1)$
\begin{equation}\label{Gc} G^{c}_{a}(n):=-\sum_{k=2}^{n}\binom{n}{k}\lambda_{n,k}\left[\left(1-\frac{k-1}{n}\right)^{1-a}-1\right],
\end{equation}
\begin{equation}\label{Gcp} G^{c,p}_{a}(n):=-\sum_{k=2}^{\lfloor np\rfloor}\binom{n}{k}\lambda_{n,k}\left[\left(1-\frac{k-1}{n}\right)^{1-a}-1\right].
\end{equation}

For any $n\geq 1$
\begin{equation}\label{Gf} G^{f}_{a}(n):=-n\sum_{k=1}^{\infty}\mu(k)\left[\left(1+\frac{k}{n}\right)^{1-a}-1\right].
\end{equation}
We need the following estimates.
\begin{lemma}\label{estimatesexplosion} Given $a>1$, 
\begin{itemize}
\item[(i)] for all $n\geq 2$ \begin{equation*} -G_{a}^{c,p}(n)\leq \frac{\Phi(n)}{n}(a-1)\left(1-p\right)^{-a},
\end{equation*}
\item[(ii)] for all $n\geq 1$ \begin{equation*} G_{a}^{f}(n)\geq 2^{-a}(a-1)\ell(n).
\end{equation*}
\end{itemize}
\end{lemma}
\begin{proof}
For any $a>1$ and any $x\in (0,1)$, by the mean-value theorem
\[(1-x)^{1-a}-1\leq (a-1)(1-x)^{-a} x.\]
Hence, for all $n\geq 2$
\begin{align*}
-G_{a}^{c,p}(n)&=\sum_{k=2}^{\lfloor np\rfloor}\binom{n}{k}\lambda_{n,k}\left[\left(1-\frac{k-1}{n}\right)^{1-a}-1\right]\\
&\leq (a-1) \left(1-\frac{\lfloor np \rfloor }{n}\right)^{-a}\sum_{k=2}^{\lfloor np \rfloor }\frac{k-1}{n} \lambda_{n,k} \binom{n}{k}\\
&\leq (a-1) \left(1-\frac{\lfloor np \rfloor }{n}\right)^{-a}\frac{\Phi(n)}{n}\leq (a-1) \left(1-p\right)^{-a}\frac{\Phi(n)}{n}.
\end{align*}
Recall now $G_a^{f}(n)$. A simple study of the function $g(x):=1-(1+x)^{1-a}-cx$ shows that if $c=(a-1)2^{-a}$ then  $1-(1+x)^{1-a}\geq cx$ for all $x\leq 1$. For all $n$,
\begin{align*}
G_a^{f}(n)&=n\sum_{k=1}^{\infty}\mu(k)\left[1-\left(1+\frac{k}{n}\right)^{1-a}\right]\\
&=n\sum_{k=1}^{n}\mu(k)\left[1-\left(1+\frac{k}{n}\right)^{1-a}\right]+n\sum_{k=n+1}^{\infty}\mu(k)\left[1-\left(1+\frac{k}{n}\right)^{1-a}\right]\\
&\geq (a-1)2^{-a}\sum_{k=1}^{n}k\mu(k)+(1-2^{1-a})n\bar{\mu}(n+1)\\
&\geq (a-1)2^{-a}\left(\sum_{k=1}^{n}k\mu(k)+n\bar{\mu}(n+1)\right)=(a-1)2^{-a}\ell(n)
\end{align*}
where we have used \eqref{ellformula} in the last equality. \qed
\end{proof}
We now deal with the proof of Lemma \ref{lemmaexplosion}.

\textbf{Proof of Lemma \ref{lemmaexplosion}}. Denote by $(N_t^{(p)},t\geq 0)$ the minimal Markov process with generator $\mathcal{L}^{p}$. We first establish that the process $(N_t^{(p)},t\geq 0)$ explodes with positive probability by applying Corollary \ref{genexplosionmaincor}. Let $c(a)=(a-1)2^{-a}$. By assumption $\mathbb{H}$, there is $n_0$ such that for all $n\geq n_0$,  $\ell(n)\geq  g(\log n)\log n$. By Lemma \ref{estimatesexplosion}-(ii), $G_a^{f}(n)\geq c(a)g(\log n)\log n $,
so that, the first assumption of Corollary \ref{genexplosionmaincor} holds. Recall now our assumption
$\rho:=\underset{n\rightarrow \infty}{\limsup}\frac{\Phi(n)}{n\ell(n)}<\frac{1}{2}$. Choose $p$ small enough such that $\frac{1}{1-p}\rho<\frac{1}{2}$. Then by Lemma \ref{estimatesexplosion}, for all $n$ and all $a>1$
\[\frac{-G_a^{c,p}(n)}{G_a^{f}(n)}\leq 2^a(1-p)^{-a}\frac{\Phi(n)}{n\ell(n)},\]
and thus, \begin{equation}\label{upperbound} \underset{n\rightarrow \infty}{\limsup} \frac{-G_a^{c,p}(n)}{G_a^{f}(n)} \leq 2^{a}\frac{\rho}{(1-p)^{a}}=:\gamma_a.
\end{equation}
By the assumption $\rho<\frac{1}{2}$, take  constant $a$ close enough to $1$ such that the upper bound $\gamma_a$ in \eqref{upperbound} above  is strictly smaller than $1$. Finally, Corollary \ref{genexplosionmaincor} applies and the process $(N_t^{(p)},t\geq 0)$ explodes with positive probability when starting from a large enough initial value. More precisely, if one denotes by $\tau_\infty^{+, p}$ the explosion time of $(N_t^{(p)},t\geq 0)$ and set $c=1-\gamma_a\in (0,\infty)$, then by applying the estimate \eqref{lowerboundexplosion}, we get that if $n_0$ is large enough,
\[\mathbb{P}_{n_0}\left(\tau_\infty^{+, p}\leq \int_{\frac{1}{1+\delta}\log n_0^{1-\delta}}^{\infty}\frac{\ddr v}{cvg(v)}\right)\geq \prod_{k=0}^{\infty}h_{a}(k,n_0)>0,\]
with $h_{a}(k,n_0):=1-(1+2^{a-1})\left(\frac{1}{n_{0}^{\delta(a-1)}}\right)^{(1+\delta)^{k}}$\!\!. Simple calculations show that $\prod_{k=0}^{\infty}h_{a}(k,n_0)$ converges to $1$ as $n_0$ goes to $\infty$. Since $\int_{\frac{1}{1+\delta}\log n_0^{1-\delta}}^{\infty}\frac{\ddr v}{cvg(v)}\underset{n_0\rightarrow \infty}{\longrightarrow} 0$, for all $t>0$
\begin{equation}\label{stoppedgoesto1}\mathbb{P}_{n_0}\left(\tau_\infty^{+, p}\leq t\right)\underset{n_0\rightarrow \infty}{\longrightarrow} 1.
\end{equation}

Recall $\sigma_p^{(n_0)}$ defined in \eqref{sigmap}. We now show that $\mathbb{P}_{n_0}(\tau_\infty^{+}<\sigma_p^{(n_0)})>0$ for large enough $n_0$, where $\tau_\infty^{+}$ denotes the first explosion time of $(\#\Pi^{(n_0)}(t),t\geq 0)$. Recall the Poisson construction of $(\Pi^{(n_0)}(t),t\geq 0)$ in Section \ref{EFCbasics}. For each atom $(t,\pi^{c})$ of $\mathrm{PPP}_C$, we associate the i.i.d Bernoulli random variables $(X_i,i\geq 1)$ defined by $X_i=1$, if $\{i\} \notin \pi^{c}$, that is to say if the $i^{\mathrm{th}}$ block of $\Pi^{(n_0)}(t-)$  takes part to the merging at time $t$,  and $X_i=0$ if $\{i\} \in \pi^c$, which means that the $i^{\mathrm{th}}$ block of $\Pi^{(n_0)}(t-)$ does not take part to the merging at time $t$. By definition, the jump time $\sigma_p^{(n_0)}$ belongs to the set of atoms of coalescence  \[J_p:=\left\{(t,\pi^{c}) \text{ atom of }\mathrm{PPP}_C;\ \exists n\geq 2; \sum_{k=1}^{n}X_k\geq np\right\}.\] Applying \cite[Lemma 3.14]{cdiEFC}, we see from the calculations in the proof of \cite[Lemma 3.15]{cdiEFC} that $\mathbb{E}(\mathrm{PPP}_C(J_p))<\infty$, hence $J_p$ is locally finite. Moreover $\Pi^{(n_0)}(0)=\Pi^{(n_0)}(0+)$ for any $n_0\geq 1$, therefore $0$ is neither a coalescence time nor an accumulation point of $J_p$ and we have $\inf\{t>0; (t,\pi^c)\in J_p\}>0$ a.s. This ensures that $\underset{n\geq 2}\inf \sigma_p^{(n)}>0$ a.s. Note that
\[\mathbb{P}_{n_0}(\tau_\infty^{+}<\sigma_p^{(n_0)})=\mathbb{P}_{n_0}(\tau_\infty^{+,p}<\sigma_p^{(n_0)})\geq \mathbb{P}_{n_0}(\tau_\infty^{+,p}<\underset{n\geq 2}{\inf} \sigma_p^{(n)}).\]
For any $t>0$, one has
\begin{align*}
\mathbb{P}_{n_0}(\tau_\infty^{+}<\sigma_p^{(n_0)})
&\geq \mathbb{P}_{n_0}(\tau_\infty^{+}\leq t,\underset{n\geq 2}{\inf} \sigma_p^{(n)}>t)\\
&= \mathbb{P}_{n_0}(\tau_\infty^{+}\leq t)+\mathbb{P}(\underset{n\geq 2}{\inf} \sigma_p^{(n)}>t)-\mathbb{P}_{n_0}(\{\tau_\infty^{+}<t\}\cup\{\underset{n\geq 2}{\inf} \sigma_p^{(n)}>t\})\\
&\geq \mathbb{P}_{n_0}(\tau_\infty^{+}\leq t)+\mathbb{P}(\underset{n\geq 2}{\inf} \sigma_p^{(n)}>t)-1.
\end{align*}
By letting $n_0$ to $\infty$ and applying \eqref{stoppedgoesto1}, we get
\[\underset{n_0\rightarrow \infty}{\lim}\mathbb{P}_{n_0}(\tau_\infty^{+}<\sigma_p^{(n_0)})\geq \mathbb{P}(\underset{n\geq 2}{\inf} \sigma_p^{(n)}>t).\]
Recall that $\underset{n\geq 2}{\inf} \sigma_p^{(n)}>0$ a.s. By letting $t$ towards $0$, we have $\underset{n_0\rightarrow \infty}{\lim}\mathbb{P}_{n_0}(\tau_\infty^{+}<\sigma_p^{(n_0)})=1$. Finally, we see that the stopped process $(\#\Pi(t\wedge \sigma_p^{(n_0)}),t\geq 0)$ explodes with a positive probability under $\mathbb{P}_{n_0}$ for large enough $n_0$. Hence $(\#\Pi(t),t\geq 0)$ started from $\#\Pi(0)=n_0$, has also a positive probability to explode for large enough $n_0$. Since  $(\#\Pi(t\wedge \tau_\infty^{+}),t\geq 0)$ is irreducible in  $\mathbb{N}$, its probability of explosion starting from $n=1$ is also positive, namely $\mathbb{P}_1(\tau_\infty^{+}<\infty)>0$. We now establish that the process $(\#\Pi(t),t\geq 0)$ started from any integer $n_0$, explodes almost-surely. Pick $t>0$ such that $\mathbb{P}_1(\tau_\infty^{+}\leq t)>0$. The stochastic monotonicity in the initial states, see Lemma \ref{monotonicity}, ensures that for any $n_0\geq 1$, $\mathbb{P}_{n_0}(\tau_\infty^{+}>t)\leq \mathbb{P}_1(\tau_\infty^{+}>t)>0$. Let $n\geq 2$, by the Markov property at time $(n-1)t$, we have
\begin{align*}\mathbb{P}_{n_0}(\tau_\infty^{+}>nt)&=\mathbb{P}_{n_0}\left(\tau_\infty^{+}>(n-1)t\right)\mathbb{E}\left(\mathbb{P}_{\#\Pi((n-1)t)}(\tau_\infty^{+}>t)\right)\\
&\leq \mathbb{P}_{n_0}\left(\tau_\infty^{+}>(n-1)t\right)\mathbb{P}_{1}\left(\tau_\infty^{+}>t\right).
\end{align*}
By induction, $$\mathbb{P}_{n_0}(\tau_\infty^{+}>nt)\leq  \mathbb{P}_{1}(\tau_\infty^{+}>t)^{n}\underset{n\rightarrow \infty}{\longrightarrow} 0.$$
Therefore, $\mathbb{P}_{n_0}(\tau^{+}_\infty<\infty)=1$.

It remains to justify that $\infty$ is an exit when $\rho<\frac{1}{4}$. Notice that if the pure coalescent part does not come down from infinity, i.e when $\sum_{n\geq 2}\frac{1}{\Phi(n)}=\infty$, then $\infty$ is necessarily an exit. We treat now the case $\sum_{n\geq 2}\frac{1}{\Phi(n)}<\infty$. Recall $\theta_{\star}$. By \cite[Lemma 4.1 and Remark 4.2]{cdiEFC}, we get
\[\theta_{\star}\geq \underset{n\rightarrow \infty}{\liminf} \frac{n\ell(n)}{\Phi(2n)}\geq  \underset{n\rightarrow \infty}{\liminf} \frac{n\ell(n)}{4\Phi(n)}\geq \frac{1}{4\rho}.\]

Finally, if $\rho<\frac{1}{4}$, then $\theta_{\star}>1$ and by Theorem  \ref{cditheorem}, $(\Pi(t),t\geq 0)$ does not come down from infinity.
\qed

We stress here on some regularity of the boundary $\infty$ when it is accessible (i.e when the process explodes). For any $n_0\in \mathbb{N}$, set $\tau^{+,(n_0)}_{\infty}:=\inf\{t; \#\Pi^{(n_0)}(t-)=\infty\}$.
\begin{lemma}\label{tregular} If the map $\ell$ satisfies condition $\mathbb{H}$ and $\rho<\frac{1}{2}$, then  the boundary $\infty$ is an instantaneous exit, namely $\tau^{+,(n_0)}_{\infty}\underset{n_0\rightarrow \infty}{\longrightarrow} 0$ a.s.
\end{lemma}
\begin{proof}
Recall $\tau_\infty^{+,p}$ the first explosion time of the unstopped process $(N_t^{(p)},t\geq 0)$ with generator $\mathcal{L}^{p}$. Lemma \ref{keylemma}, and the  estimates 
\eqref{lowerboundexplosion} with $\delta<1$. Set $\varphi(n_0)=\int_{\frac{1}{1+\delta}\log n_0^{1-\delta}}^{\infty}\frac{\ddr v}{cvg(v)}$. One has $\varphi(n_0) \underset{n_0\rightarrow \infty}{\longrightarrow} 0$ and one easily checks from \eqref{lowerboundexplosion} that $\mathbb{P}(\tau_\infty^{+,p}\leq \varphi(n_0)) \underset{n_0\rightarrow \infty}{\longrightarrow} 1$. Recall $\sigma_p:=\underset{n\geq 1}{\inf} \sigma_{p}^{n}>0$ a.s. Fix $t>0$, and take $n_0$  large enough such that $\varphi(n_0)\leq t$. Then
\begin{align*}
\mathbb{P}_{n_0}(\tau_\infty^{+}\leq t)&\geq \mathbb{P}_{n_0}(\tau_\infty^{+}\leq \varphi(n_0), \varphi(n_0)<\sigma_p)\\
&=\mathbb{P}_{n_0}(\tau_\infty^{+,p}\leq \varphi(n_0) , \varphi(n_0)<\sigma_p))\\
&\geq \mathbb{P}_{n_0}(\tau_\infty^{+,p}\leq \varphi(n_0))+\mathbb{P}(\varphi(n_0)<\sigma_p)-1.
\end{align*}
Since $\sigma_p>0$ a.s,  $\mathbb{P}(\varphi(n_0)<\sigma_p)\underset{n_0\rightarrow \infty}{\longrightarrow} 1$ and
$\mathbb{P}_{n_0}(\tau_\infty^{+}\leq t)\underset{n_0\rightarrow \infty}{\longrightarrow} 1$ for any fixed $t>0$.\qed
\end{proof}

\subsection{Sufficient conditions for non-explosion}\label{nonexplosec}\

We first establish Theorem \ref{suffcondpropentrance} and then design, in a similar way as what has been done for the explosion, some sufficient conditions for non-explosion using the Lyapunov function $g_a(n)=n^{1-a}$ with $a<1$.\\

\textbf{Proof of Theorem \ref{suffcondpropentrance}}.
Recall the statement of Theorem \ref{suffcondpropentrance}.
Recall also that the sequence $(\Phi(n)/n, n\geq 2)$ is non-decreasing, see Section \ref{CDIcoal}. Assume first that $\underset{n\rightarrow \infty}{\lim} \frac{n}{\Phi(n)}>0$. By assumption $\sum_{n\geq 2}\frac{n}{\Phi(n)}\bar{\mu}(n)<\infty$ and therefore $\sum_{n=2}^{\infty}\bar{\mu}(n)<\infty$. In this case, the process $(\#\Pi(t),t\geq 0)$ is clearly non-explosive since it stays below a branching process whose offspring measure $\mu$ has finite mean. 

We now treat the case for which $\underset{n\rightarrow \infty}{\lim} \frac{n}{\Phi(n)}=0$. 
Set $f(n):=\sum_{k=2}^{n}\frac{k}{\Phi(k)}$ for all $n\geq 2$ and $f(1)=\frac{2}{\Phi(2)}$. Since $\Phi(k)\leq \frac{\Lambda([0,1])}{2}k^2$ for all $k\geq 2$, see Section \ref{CDIcoal}, 
one has $\frac{k}{\Phi(k)}\geq \frac{2}{\Lambda([0,1])k}$ for all $k\geq 2$ and then
$f(n)\underset{n\rightarrow \infty}{\longrightarrow} \infty$. For any $n\geq 2$,
\begin{align}\label{Lf}
\mathcal{L}^{f}f(n)&=n\sum_{k=2}^{\infty}\mu(k)\sum_{j=n+1}^{n+k}\frac{j}{\Phi(j)}=n\sum_{j=n+1}^{\infty}\sum_{k=j-n}^{\infty}\frac{j}{\Phi(j)}\mu(k)=n\sum_{i=1}^{\infty}\frac{i+n}{\Phi(i+n)}\bar{\mu}(i).
\end{align}
For the coalescent part, since $\frac{n}{\Phi(n)}\leq \frac{j}{\Phi(j)}$ for $j\leq n$
\begin{align}\label{Lc}
\mathcal{L}^{c}f(n)
&=-\sum_{k=2}^{n}\binom{n}{k}\lambda_{n,k}\sum_{j=n-k+2}^{n}\frac{j}{\Phi(j)}\leq -\frac{n}{\Phi(n)}\sum_{k=2}^{n}\binom{n}{k}\lambda_{n,k}(k-1)=-n.
\end{align}
We now check that there exists $c$ such that $\mathcal{L}f(n)\leq cf(n)$ for all $n\geq 1$. By combining \eqref{Lf} and \eqref{Lc} one gets for any $n\geq 2$  $$\mathcal{L}f(n)\leq n\left(-1+\sum_{i=1}^{\infty}\frac{i+n}{\Phi(i+n)}\bar{\mu}(i)\right).$$
Recall the assumptions $\sum_{i\geq 2}\frac{i}{\Phi(i)}\bar{\mu}(i)<\infty$ and $\underset{n\rightarrow \infty}{\lim}\frac{n}{\Phi(n)}=0$.  Since $\frac{i+n}{\Phi(i+n)}\leq \frac{i}{\Phi(i)}$ for any $n,i\geq 2$, by Lebesgue's theorem, $\sum_{i=1}^{\infty}\frac{i+n}{\Phi(i+n)}\bar{\mu}(i)\underset{n\rightarrow \infty}{\longrightarrow} 0$. Therefore, there exists $n_{0}$, such that $\mathcal{L}f(n)\leq 0$ for all $n\geq n_{0}$. This entails that for all $n\geq 1$, $\mathcal{L}f(n)\leq c_0$
with $c_0:=\underset{k\in [|1,n_{0}|]}{\max}(|\mathcal{L}f(k)|)$. By setting $c=\frac{\Phi(2)}{2}c_0$, since $f(n)\geq \frac{2}{\Phi(2)}$ for all $n\geq 1$, we get finally $\mathcal{L}f(n)\leq cf(n)$ for all $n\geq 1$. As recalled in Section \ref{lyapunovcond}, this entails that the process does not explode.

According to \cite[Corollary 4-(2)]{cdiEFC}, if $\sum_{n=1}^{\infty}\frac{1}{\Phi(n)}<\infty$ and $\sum_{n=1}^{\infty}\frac{n}{\Phi(n)}\bar{\mu}(n)<\infty$, then $\theta=0$ and by Theorem \ref{cditheorem} the process comes down from infinity. We conclude that $\infty$ is an entrance.\qed


We establish now sufficient conditions for $\infty$ to be inaccessible, based on Proposition \ref{genmaincornonexplosion}. We assume in this section that $\sum_{k=n+1}^{\infty}k^{1-a}\mu(k)<\infty$. for some $a<1$. This ensures that $G_a^{f}$ in \eqref{Gf} is well-defined. We first find some estimates of the functions $G_a^{f}$ and $G_a^{c}$.
Recall that for all $a<1$ and all $n\geq 1$ $G_{a}^{c}(n)\geq 0$ and $G_{a}^{f}(n)\leq 0$ for all $n\geq 1$.
\begin{lemma}\label{estimateGaentrance} Given $a<1$ such that $\sum_{k=n+1}^{\infty}k^{1-a}\mu(k)<\infty$,
\begin{itemize}
\item[(1)] for all $n\geq 2$, \[G_a^{c}(n)\geq (1-a)\frac{\Phi(n)}{n},\]
\item[(2)] for all $n\geq 1$, \[-G_a^{f}(n)\leq (1-a)\sum_{k=1}^{n}k\mu(k)+n^{a}\sum_{k=n+1}^{\infty}k^{1-a}\mu(k).\]
\end{itemize}
\end{lemma}
\begin{proof} Set $g(x)=1-(1-x)^{1-a}-(1-a)x$ for all $x\in (0,1)$. A simple study of the function $g$ yields that $g(x)\geq 0$ and thus $1-(1-x)^{1-a}\geq (1-a)x$ for all $x\in(0,1)$.
Hence, by definition of $\Phi$, for all $n\geq 2$
\[G_a^{c}(n)=\sum_{k=2}^{n}\binom{n}{k}\lambda_{n,k}\left[1-\left(1-\frac{k-1}{n}\right)^{1-a}\right]\geq (1-a)\frac{\Phi(n)}{n}.\]
Recall $G_a^{f}$ in \eqref{Gf}. Note that for any $x\in (0,1)$, $(1+x)^{1-a}-1\leq (1-a)x$
and for any $x\in (1,\infty)$, $(1+x)^{1-a}-1\leq x^{1-a}$. We get for all $n\geq 2$
\begin{align*}
-G_a^f(n)&=n\sum_{k=1}^{\infty}\mu(k)\left[\left(1+\frac{k}{n}\right)^{1-a}-1\right]\leq (1-a)\sum_{k=1}^{n}k\mu(k)+n\sum_{k=n+1}^{\infty}\mu(k)\left(\frac{k}{n}\right)^{1-a}.
\end{align*}
\qed
\end{proof}

\begin{proposition}\label{suffcondpropentrance1} Assume that there exists $a\in (0,1)$ such that $\sum_{n=1}^{\infty}n^{1-a}\mu(n)<\infty$. Set the condition \begin{equation}\label{secondcond} \underset{n\rightarrow \infty}{\lim} \frac{n^{1+a}}{\Phi(n)}\sum_{k=n}^{\infty}k^{1-a}\mu(k)= 0.
\end{equation}
If \eqref{secondcond} holds and  $\underset{n\rightarrow \infty}{\limsup}\frac{n}{\Phi(n)}\sum_{k=1}^{n}k\mu(k)<1$,  then the process $(\#\Pi(t),t\geq 0)$ does not explode.
%
\end{proposition}

We shall see an example where Proposition \ref{suffcondpropentrance1} applies when establishing Theorem \ref{logcoal} in Section \ref{slowcoalsec}.

\begin{proof} By Lemma \ref{estimateGaentrance},
\begin{equation}\label{bound}\underset{n\rightarrow \infty}{\limsup}\frac{-G^{f}_a(n)}{G_a^{c}(n)}\leq \underset{n\rightarrow \infty}{\limsup} \frac{n}{\Phi(n)}\sum_{k=1}^{n}k\mu(k)+\frac{1}{1-a}\underset{n\rightarrow \infty}{\limsup} \frac{n^{a+1}}{\Phi(n)} \sum_{k=n+1}^{\infty}k^{1-a}\mu(k).
\end{equation} 	
By the assumption \eqref{secondcond}, the second term on the right-hand side of \eqref{bound} vanishes and we get
\[\underset{n\rightarrow \infty}{\limsup}\frac{-G^{f}_a(n)}{G_a^{c}(n)}\leq \underset{n\rightarrow \infty}{\limsup} \frac{n}{\Phi(n)}\sum_{k=1}^{n}k\mu(k).\]
Proposition \ref{genmaincornonexplosion} applies and yields that the process does not explode.
\qed
\end{proof}

%
%
\subsection{Regularly varying coagulation/fragmentation mechanisms}\

In this Section, we will study into more details the cases
\begin{center}
$\Phi(n)\underset{n\rightarrow \infty}{\sim} dn^{\beta +1}$ and $\mu(n)\underset{n\rightarrow \infty}{\sim} \frac{b}{n^{\alpha+1}}$.
\end{center}
We establish Theorem \ref{stablefragtheorem} with several lemmas.
\begin{lemma}\label{part1proofregularvarytheorem} Let $\alpha \in (0,\infty)$ and $\beta\in (0,1)$. One has the following classification.
\begin{itemize}
\item If $\alpha+\beta<1$, then $(\#\Pi(t),t\geq 0)$ explodes almost-surely and $\infty$ is an exit boundary.
\item If $\alpha+\beta>1$, then $(\#\Pi(t),t\geq 0)$ does not explode almost-surely and $\infty$ is an entrance boundary.
\end{itemize}
\end{lemma}
\begin{proof}
The cases $\alpha+\beta<1$ and $\alpha+\beta>1$ respectively are consequences of Theorem \ref{suffcondpropexit} and Theorem \ref{suffcondpropentrance}, respectively. \qed
\end{proof}
It remains to complete the proof of Theorem \ref{stablefragtheorem} when $\alpha=1-\beta$. In this case, the sufficient conditions for explosion or non explosion obtained in Section \ref{explosec} and Section \ref{nonexplosec} do not allow us to conclude. We provide a finer study of $G_a^{f}$ for this case in the following lemma.

\begin{lemma}\label{estimateGafstable} Let $\alpha\in (0,1)$ and $b_2>b_1>0$. Assume that the splitting measure $\mu$ satisfies for all large enough $k$, \begin{equation}\label{assumptionfrag} \frac{b_1}{k^{\alpha+1}}\leq \mu(k)\leq \frac{b_2}{k^{\alpha+1}}.
\end{equation}

For any $a>1$ and for any $\epsilon>0$, there is $n_0$ such that if $n\geq n_0$,
\begin{equation}\label{lowerboundfrag}
G^{f}_{a}(n)\geq \frac{1}{1+\epsilon} i_\alpha(a)b_1n^{1-\alpha}
\end{equation}
where $i_{\alpha}(a):=\int_{0}^{\infty}\frac{1-(1+u)^{1-a}}{u^{\alpha+1}}\ddr u<\infty$.
\\

For any $a\in (1-\alpha,1)$ and for any $\epsilon>0$, there is $n_0$ such that if $n\geq n_0$,
\begin{equation}\label{upperboundfragstable}
-G^{f}_{a}(n)\leq \frac{1}{1-\epsilon}j_\alpha(a)b_2n^{1-\alpha}
\end{equation}
where $j_{\alpha}(a):=\int_{0}^{\infty}\frac{(1+u)^{1-a}-1}{u^{\alpha+1}}\ddr u<\infty$.
\end{lemma}

\begin{proof}
Let $a>1$. For any $n\geq 1$,
\begin{align*}
n\sum_{k=1}^{\infty}\frac{1}{k^{\alpha +1}}\left[1-\left(1+\frac{k}{n}\right)^{1-a}\right]&= n\sum_{k=2}^{\infty}\frac{1}{k^{\alpha +1}}\left[1-\left(1+\frac{k}{n}\right)^{1-a}\right]+n\left(1-(1+1/n)^{1-a}\right)
\end{align*}
and
\begin{align*}
n\sum_{k=2}^{\infty}\frac{1}{k^{\alpha +1}}\left[1-\left(1+\frac{k}{n}\right)^{1-a}\right] 	&\underset{n\rightarrow \infty}{\sim } n\sum_{k=2}^{\infty} \int_{k-1}^{k}\frac{1}{x^{\alpha +1}}\left[1-\left(1+\frac{x}{n}\right)^{1-a}\right]\ddr x\\
& \underset{n\rightarrow \infty}{\sim } n\int_{1}^{\infty}\frac{1}{x^{\alpha +1}}\left[1-\left(1+\frac{x}{n}\right)^{1-a}\right]\ddr x\\
& \underset{n\rightarrow \infty}{\sim } n^{1-\alpha}\int_{1/n}^{\infty}\frac{1-(1+u)^{1-a}}{u^{\alpha+1}}\ddr u \underset{n\rightarrow \infty}{\sim} i_\alpha(a) n^{1-\alpha},
\end{align*}
with $i_{\alpha}(a):=\int_{0}^{\infty}\frac{1-(1+u)^{1-a}}{u^{\alpha+1}}\ddr u<\infty$.
In particular, we see that for any $\epsilon>0$, there is $n_0$ such that if $n\geq n_0$,
\eqref{lowerboundfrag} holds. The proof is similar for the case $a<1$.\qed
\end{proof}

We will be able to study the explosion in the case of a regularly coagulation mechanism $\Phi(n)\underset{n\rightarrow \infty}{\sim} dn^{\beta+1}$ when $\alpha+\beta=1$ by applying Lemma \ref{estimateGafstable}. We need first an analytical lemma. Recall functions $i_\alpha$ and $j_\alpha$ defined in Lemma \ref{estimateGafstable}.

\begin{lemma}\label{ineq} For all $\alpha \in (0,1)$, set $I(\alpha):=\int_{0}^{\infty}\frac{\log(1+u)}{u^{\alpha+1}}\ddr u$. Then  for all $\alpha\in (0,1)$, $I(\alpha)=\frac{\pi}{\alpha \sin(\pi \alpha)}>\frac{1}{\alpha(1-\alpha)}$. Further, 
\[\frac{i_{\alpha}(a)}{a-1}\underset{a\rightarrow 1^{+}}\longrightarrow I(\alpha) \text{\,\, and \,\, } \frac{j_{\alpha}(a)}{a-1}\underset{a\rightarrow 1^{-}}\longrightarrow I(\alpha).\]

\end{lemma}

\begin{proof} We first establish the convergence statements. Recall $i_\alpha(a)=\int_{0}^{\infty}\frac{1-(1+u)^{1-a}}{u^{\alpha+1}}\ddr u$.  For any $a>1$ and any $u>0$, \[1-(1+u)^{1-a}=1-e^{-(a-1)\log (1+u)}\leq (a-1)\log(1+u).\] By Lebesgue's theorem, we have  $\frac{i_\alpha(a)}{a-1}\underset{a\rightarrow 1}{\longrightarrow} I(\alpha):=\int_{0}^{\infty}\frac{\log(1+u)}{u^{\alpha+1}}\ddr u$.

Recall $j_\alpha(a)=\int_{0}^{\infty}\frac{(1+u)^{1-a}-1}{u^{\alpha+1}}\ddr u<\infty$, for any $1-\alpha<a<1$.
Let $\epsilon>0$ and assume $a>1-\alpha+\epsilon$. By applying the mean value theorem to the function  $a\mapsto (1+u)^{1-a}-1= e^{(1-a)\log(1+u)}-1$, we get that for any $u>0$  \[(1+u)^{1-a}-1\leq (1-a)\log(1+u)(1+u)^{1-a}\leq (1-a)\log(1+u)(1+u)^{\alpha-\epsilon}.\]
Thus, $\frac{1}{1-a}\frac{1-(1+u)^{1-a}}{u^{\alpha+1}}\leq \log(1+u)\frac{(1+u)^{\alpha-\epsilon}}{u^{\alpha+1}}$. The function $u\mapsto \log(1+u)\frac{(1+u)^{\alpha-\epsilon}}{u^{\alpha+1}}$ is integrable on $(0,\infty)$ and by Lebesgue's theorem, we have  $\frac{j_\alpha(a)}{1-a}\underset{a\rightarrow 1^{-}}{\longrightarrow} I(\alpha)$.

We now show the identity $I(\alpha)=\frac{\pi}{\alpha \sin(\pi \alpha)}$ for any $\alpha\in (0,1)$. One has, by integration by parts and the change of variable $v=\frac{1}{1+u}$
\begin{align*}
I(\alpha):=\int_{0}^{\infty}\frac{\log(1+u)}{u^{1+\alpha}}\ddr u&=\frac{1}{\alpha}\int_{0}^{\infty}\frac{1}{1+u}u^{-\alpha}\ddr u\\
&=\frac{1}{\alpha}\int_{0}^{1}\frac{1}{v}\left(\frac{1}{v}-1\right)^{-\alpha}\ddr v\\
&=\frac{1}{\alpha}\int_{0}^{1}v^{\alpha-1}(1-v)^{-\alpha}\ddr v\\
&=\frac{1}{\alpha}\mathrm{Beta}(\alpha,1-\alpha)=\frac{1}{\alpha}\Gamma(1-\alpha)\Gamma(\alpha)=\frac{1}{\alpha}\frac{\pi}{\sin(\pi \alpha)}
\end{align*}
where $\mathrm{Beta}(\cdot,\cdot)$ is the Beta function. We have used Beta-Gamma relation and Euler's reflection formula in the two last equalities, see e.g. \cite[VIII.3]{Schwartz}. The strict inequality $\frac{\pi}{\alpha \sin(\pi \alpha)}>\frac{1}{\alpha(1-\alpha)}$ can be easily checked by showing that the function $\alpha\mapsto (1-\alpha)\pi-\sin(\pi \alpha)$ is strictly decreasing on $(0,1]$.\qed
\end{proof}
%

\begin{lemma}\label{explosioninstable} Assume that $\mu(n)\underset{n\rightarrow \infty}{\sim} \frac{b}{n^{\alpha+1}}$, $\Phi(n)\underset{n\rightarrow\infty}{\sim} dn^{\beta +1}$ and  $\alpha+\beta=1$. For any $n\in \mathbb{N}$, under $\mathbb{P}_n$,
\begin{enumerate}
\item if $\frac{d}{b}<I(\alpha)$, then $(\#\Pi(t),t\geq 0)$ explodes almost-surely,
\item if $\frac{d}{b}>I(\alpha)$, then $(\#\Pi(t),t\geq 0)$ does not explode almost-surely.
\end{enumerate}
\end{lemma}
\begin{proof} We establish assertion (1).  Assume that $\frac{1}{I(\alpha)}\frac{d}{b}<1$. Let $p$ be small enough such that $\frac{1}{1-p}\frac{1}{I(\alpha)}\frac{d}{b}<1$. Recall that $(N_t^{(p)},t\geq 0)$ denotes the process $(\#\Pi(t),t\geq 0)$ stopped at $\sigma_p$. We show that $(N_t^{(p)},t\geq 0)$ explodes with positive probability by using Corollary \ref{genexplosionmaincor}. Using Lemma \ref{estimateGafstable}, we see that the first condition on $G_a^{f}$ is plainly satisfied with for instance $g(n):=ce^{(1-\alpha)n}/n$ for some constant $c>0$. According to Lemma \ref{estimatesexplosion}-(i), \[-G_a^{c,p}(n)\leq \frac{\Phi(n)}{n}(a-1)\left(1-p\right)^{-a}.\]
Combining this latter bound with Lemma \ref{estimateGafstable}, we get that
\[\underset{n\rightarrow \infty}{\limsup} \frac{-G_a^{c,p}(n)}{G_a^{f}(n)}\leq (1-p)^{-a}\frac{a-1}{i_\alpha(a)}\underset{n\rightarrow \infty}{\limsup} \frac{\Phi(n)}{bn^{2-\alpha}}.\]
Since $\Phi(n)\underset{n\rightarrow \infty}{\sim} dn^{2-\alpha}$, we have
\begin{equation}\label{upperboundstable} \underset{n\rightarrow \infty}{\limsup} \frac{-G_a^{c,p}(n)}{G_a^{f}(n)}\leq \frac{1}{(1-p)^{a}}\frac{a-1}{i_\alpha(a)}\frac{d}{b}=:\gamma_{a,p}.
\end{equation}
The upper bound in \eqref{upperboundstable}, $\gamma_{a,p}$, converges towards $\frac{1}{1-p}\frac{1}{I(\alpha)}\frac{d}{b}<1$ as $a$ goes towards $1^{+}$. We can therefore find $a>1$ close enough to $1$ such that $\underset{n\rightarrow \infty}{\limsup} \frac{-G_a^{c,p}(n)}{G_a^{f}(n)}<1$. The fact that the unstopped process $(\#\Pi(t),t\geq 0)$ explodes almost-surely is proven by the same argument as in the proof of Lemma \ref{lemmaexplosion}.

We now establish (2).
%
Recall Lemma \ref{estimateGaentrance} and the bound $G^{c}_{a}(n)\geq (1-a)\frac{\Phi(n)}{n}$ for all $n$. Then \[\underset{n\rightarrow \infty}{\limsup} \frac{-G_a^{f}(n)}{G_a^{c}(n)}\leq \frac{j_{\alpha}(a)}{1-a}\frac{b}{d}.\]
The upper bound goes to $\frac{b}{d}I(\alpha)$ as $a$ goes to $1$. Thus, if $\frac{b}{d}I(\alpha)< 1$, one can find $a<1$ close enough to $1$ such that
\[\underset{n\rightarrow \infty}{\limsup} \frac{-G_a^{f}(n)}{G_a^{c}(n)}<1.\]
By Proposition \ref{genmaincornonexplosion}, the process $(\#\Pi(t),t\geq 0)$ does not explode.  \qed
\end{proof}

We now classify the possible behaviors of the process $(\Pi(t),t\geq 0)$ on the boundary of partitions with infinitely many blocks, by combining the properties of explosion and of coming down from infinity of the process $(\#\Pi(t),t\geq 0)$.

\begin{lemma}\label{part2proofregularvarytheorem} Assume $\mu(n)\underset{n\rightarrow \infty}{\sim}\frac{b}{n^{\alpha+1}}$ with $b>0$ and $\Phi(n)\underset{n\rightarrow \infty}{\sim} dn^{\beta+1}$ with $\beta=1-\alpha$, $d>0$. Set $\sigma:=\frac{b}{d}\frac{\pi}{\alpha \sin(\pi \alpha)}$ and recall $\theta:=\frac{b}{d}\frac{1}{\alpha(1-\alpha)}$ defined in Proposition \ref{regularcdi}. One has $\sigma>\theta$ for all $\alpha \in (0,1)$ and
\begin{itemize}
\item if $\theta>1$ then $(\#\Pi(t),t\geq 0)$ explodes and stays infinite almost-surely: the boundary $\infty$ is an exit,
\vspace*{1mm}
\item if $\theta<1<\sigma$ then $(\#\Pi(t),t\geq 0)$ explodes and comes down from infinity almost-surely: the boundary $\infty$ is regular,
\vspace*{1mm}
\item if $\sigma<1$ then $(\#\Pi(t),t\geq 0)$ does not explode and comes down from infinity almost-surely: the boundary $\infty$ is an entrance.
\end{itemize}
\end{lemma}
\begin{proof}
First we note that if $\mu(n)\underset{n\rightarrow \infty}{\sim}\frac{b}{n^{\alpha+1}}$ then $\bar{\mu}(n)\underset{n\rightarrow \infty}{\sim} \frac{\lambda}{n^{\alpha}}$ with $\lambda=b/\alpha$. Recall Proposition \ref{regularcdi} and Theorem \ref{cditheorem}. When $\theta=\frac{\lambda}{d(1-\alpha)}>1$ (respectively, $\theta<1$), the process stays infinite (respectively, comes down from infinity).  The strict inequality in Lemma \ref{ineq} ensures that $\sigma> \theta$. In particular, we see that if $\theta>1$, then $\sigma>1$ which entails on the one hand that the process cannot leave infinity, and on the other hand, by Lemma \ref{explosioninstable}, that started from a finite state, it explodes almost-surely. When $\sigma>1>\theta$, by Lemma \ref{explosioninstable}, $(\#\Pi(t),t\geq 0)$ explodes a.s and by Proposition \ref{regularcdi}, it comes down from infinity a.s, thus $\infty$ is regular. \qed
\end{proof}
\noindent \textbf{Proof of Theorem \ref{stablefragtheorem}.} The first two statements of Theorem \ref{stablefragtheorem} are provided by Lemma \ref{part1proofregularvarytheorem}. The third statement is provided by Lemma \ref{part2proofregularvarytheorem}. Note that $\sigma>1$ is equivalent to $\frac{b}{d}>\frac{\alpha \sin \pi \alpha }{\pi}$ and $\theta<1$ is equivalent to $\frac{b}{d}<\alpha(1-\alpha)$. \qed

It remains to show that when the boundary is regular, it is regular for itself. The following lemma establishes Proposition \ref{regularforitself}.
\begin{lemma}\label{regularforitselflemma} Consider a simple EFC process $(\Pi(t),t\geq 0)$ with coagulation and splitting measures satisfying the assumptions of Lemma \ref{part2proofregularvarytheorem}. Assume that $\#\Pi(0)=\infty$ and recall $\tau_\infty^{+}:=\inf\{t>0;\#\Pi(t-)=\infty\}$. If $\theta<1<\sigma$ then, \[\mathbb{P}(\tau^{+}_{\infty}=0)=1.\]
\end{lemma}
\begin{proof}
Assume first that $\Pi(0)$ is proper. We show that explosion occurs instantaneously almost-surely by considering the estimates \eqref{lowerboundexplosion} for the first explosion time. As noticed in the proof of Lemma \ref{explosioninstable}, 
the condition $\mathbb{H}$ is fulfilled with  
the function $g(x)=e^{(1-\alpha)x}/x$ for all $x\geq 1$. For any $n\in \mathbb{N}$, recall the process $(\Pi^{(n)}(t),t\geq 0)$ defined in Section \ref{EFCbasics} and that $(\#\Pi^{(n)}(t),t<\tau_\infty^{+})$ has the same law as $(\#\Pi(t),t<\tau_\infty^{+})$ when $\#\Pi(0)=n$. Let $n_0\in \mathbb{N}$, and recall $\sigma^{(n_0)}_p$ defined in \eqref{sigmap}.
We have seen in the proof of Lemma \ref{lemmaexplosion} that if $\sigma>1$, then one can find  $p$ small enough and $a$ close enough to $1$, such that $\gamma_{a,p}<1$, where $\gamma_{a,p}$ is defined in \eqref{upperboundstable}.
By setting $c:=1-\gamma_{a,p}>0$ and applying the estimate \eqref{lowerboundexplosion}, we obtain that for any $0<\delta<\frac{1}{2a-1}$,
\[\mathbb{P}_{n_0}\left(\tau^{+,(n_0)}_\infty\leq \varphi(n_0), \tau^{+,(n_0)}_\infty\leq \sigma_p^{(n_0)}\right)\underset{n_0\rightarrow \infty}{\longrightarrow} 1\]
with $\tau_{\infty}^{+,(n_0)}:=\inf\{t>0; \#\Pi^{(n_0)}(t-)=\infty\}$ and $\varphi(n_0)=\int_{\frac{1}{1+\delta}\log n_0^{1-\delta}}^{\infty}\frac{\ddr v}{cvg(v)}.$ According to Lemma \ref{monotonicity}, for all $n_0\geq 1$ and all $t\geq 0$, $\#\Pi^{(n_0)}(t)\leq \#\Pi(t)$ almost-surely, thus $\mathbb{P}(\tau_{+}^{\infty}\leq \tau_+^{\infty,(n_0)})=1$.
Since $\varphi(n_0)\underset{n_0\rightarrow \infty}{\longrightarrow} 0$, we have that
\begin{equation}\mathbb{P}\left(\tau^{+}_\infty=0\right)\geq \underset{n_0\rightarrow \infty}{\lim}\mathbb{P}\left(\tau_\infty^{+,n_0}\leq \varphi(n_0)\right)=1.
\end{equation}
If $\Pi(0)$ is improper, then for any $t>0$, set $\tau_\infty^{+}(t):=\inf\{s>0; \#\Pi((t+s)-)=\infty\}$. One has $\tau_\infty^{+}(t)=\tau_\infty^{+, (\#\Pi(t-))}$ in law and since $\#\Pi(t-)\underset{t\rightarrow 0}{\longrightarrow} \infty$ a.s. our previous argument entails that $\tau_\infty^{+}(t)$ goes to $0$ in probability. Since $\tau_\infty^{+}\leq \tau_\infty^{+}(t)+t$, we get $\tau_\infty^{+}=0$ a.s.
\qed
\end{proof}
We study now the critical case for which $\beta=1-\alpha$ and $\sigma=1$ i.e. $\frac{b}{d}=\frac{\alpha \sin(\pi \alpha)}{\pi}$, and establish Proposition \ref{criticalvalue}.\\

\noindent \textbf{Proof of Proposition \ref{criticalvalue}}.
Recall the assumptions on $\mu$ and $\Lambda$ and the maps $\Phi$ and $\Psi$ in \eqref{phi2} and \eqref{psi}. We first show that under these assumptions $\Phi(n)\geq dn^{1+\beta}-Cn$ for large $n$ and some constant $C>0$. Note that for any $n\geq 2$, $\Psi(n)-\Phi(n)\leq  C'n$ for some constant $C'>0$, see e.g \cite[Lemma 2.1]{limic2015}. Moreover, recalling that for any $n\geq 0$, $\int_{0}^{\infty}(e^{-nx}-1+nx)x^{-2-\beta}\ddr x=\frac{\beta(\beta+1)}{\Gamma(1-\beta)}n^{1+\beta}$, we get
\[\tilde{\Psi}(n):=d\frac{\Gamma(1-\beta)}{\beta(\beta+1)}\int_{0}^{\infty}(e^{-nx}-1+nx)x^{-2-\beta}\left(h(x)\mathbbm{1}_{[0,1]}(x)+\mathbbm{1}_{]1,\infty[}(x)\right)\ddr x\geq dn^{1+\beta}.\] 
For any $n\geq 2$,
\begin{align*}
dn^{1+\beta}-\Phi(n)\leq \tilde{\Psi}(n)-\Phi(n)&=\tilde{\Psi}(n)-\Psi(n)+\Psi(n)-\Phi(n)\\
&= d\frac{\Gamma(1-\beta)}{\beta(\beta+1)}\int_{1}^{\infty}(e^{-nx}-1+nx)x^{-2-\beta}\ddr x +\Psi(n)-\Phi(n)\\
&\leq\left( d\frac{\Gamma(1-\beta)}{\beta(\beta+1)}\int_{1}^{\infty}x^{-1-\beta}\ddr x+C'\right)n\leq Cn.
\end{align*}
We show now that $\mathcal{L}\log(n+1)\leq c\log(n+1)$ for some $c>0$ and $n\geq 1$. Non-explosion will follow by applying Foster-Lyapunov criterion recalled in Section \ref{lyapunovcond}.
Let $n\geq 2$,
\begin{align*}
\mathcal{L}^{c}\log n&=\sum_{k=2}^{n}\log\left( \frac{n-k+1}{n}\right)\binom{n}{k}\lambda_{n,k}\\
&=\sum_{k=2}^{n}\log\left( 1-\frac{k-1}{n}\right)\binom{n}{k}\lambda_{n,k}\leq -\sum_{k=2}^{n}\frac{k-1}{n}\binom{n}{k}\lambda_{n,k}=-\frac{\Phi(n)}{n}\leq -dn^{\beta}+C.
\end{align*}
Recall $I(\alpha):=\int_{0}^{\infty}\frac{b}{x^{1-\alpha}}\log\left(1+x\right)\ddr x=\frac{\pi}{\alpha\sin(\pi \alpha)}$. For any $n\geq 1$, 
\begin{align*}
\mathcal{L}^{f}\log n&=n\sum_{k=1}^{\infty}\frac{b}{k^{1-\alpha}}\log\left(1+\frac{k}{n}\right)\leq nb\log\left(1+\frac{1}{n}\right)+n\sum_{k=2}^{\infty}\int_{k-1}^{k}\frac{b}{x^{1-\alpha}}\log\left(1+\frac{x}{n}\right)\ddr x\\
&\leq b+n\int_{1}^{\infty}\frac{b}{x^{1-\alpha}}\log\left(1+\frac{x}{n}\right)\ddr x\leq b+bn^{1-\alpha}I(\alpha).
\end{align*}
Finally, since $\frac{b}{d}=1/I(\alpha)$ we obtain for any $n\geq 1$,
\[\mathcal{L}\log (n+1)\leq b+C+(bI(\alpha)-d)(n+1)^{1-\alpha}= b+C\leq c\log(n+1)\]
with $c:=\frac{b+C}{\log 2}$. 

By applying Proposition \ref{regularcdi}, we see that $\theta:=\frac{b}{d}\frac{1}{\alpha(1-\alpha)}$. Lemma \ref{ineq} entails that $\theta:=\frac{b}{d}\frac{1}{\alpha(1-\alpha)}<\sigma:=\frac{b}{d}\frac{\pi}{\alpha \sin(\pi \alpha)}$, so that if $\frac{b}{d}=\frac{1}{I(\alpha)}=\frac{\alpha \sin(\pi \alpha)}{\pi}$, then $\sigma=1$, $\theta<1$ and by Proposition \ref{regularcdi}, the process $(\#\Pi(t),t\geq 0)$ comes down from infinity. \qed

\subsection{Slower coalescence}\label{slowcoalsec}\

We now study the case of  a coagulation measure satisfying $\Phi(n)\underset{n\rightarrow \infty}{\sim} d n(\log n)^{\beta}$ for some $d>0$ and $\beta>1$ and establish Theorem \ref{logcoal}. Note that the  coalescences  occur slower than that  in the previous section. Let $b>0$ and $\alpha>0$. Consider a splitting measure $\mu$ satisfying  $\mu(n)\underset{n\rightarrow  \infty}{\sim} b \frac{(\log n)^{\alpha}}{n^2}$. In this case, its tail asymptotics are of the form

\begin{center} $\bar{\mu}(n)\underset{n\rightarrow \infty}{\sim} b \frac{(\log n )^{\alpha}}{n}$ and $\sum_{k=1}^{n}k\mu(k) \underset{n\rightarrow \infty}{\sim} \frac{b}{\alpha+1}(\log n)^{\alpha+1}$. \end{center} Recalling formula \eqref{ellformula}, we get $\ell(n)  \underset{n\rightarrow \infty}{\sim} \frac{b}{\alpha+1}(\log n)^{\alpha+1}$.

\begin{lemma}\label{logcoallemma} One has the following classification
\begin{itemize}
\item if $\beta<1+\alpha$, then $(\#\Pi(t),t\geq 0)$ explodes and stays infinite almost-surely:
$\infty$ is an exit boundary,
\vspace*{1mm}
\item if $\beta>1+\alpha$, then $(\#\Pi(t),t\geq 0)$ does not explode  and when started from a partition with infinitely many blocks, comes down from infinity instantaneously almost-surely: $\infty$ is an entrance boundary,
\end{itemize}
\end{lemma}
\begin{proof}
The cases $\beta<1+\alpha$ and $\beta>1+\alpha$ respectively are consequences of Theorem \ref{suffcondpropexit} and Theorem \ref{suffcondpropentrance}, respectively. See  Example \ref{exampleexit}-(2) and Example \ref{exampleentrance}-(2). \qed
\end{proof}

Set $\theta=\frac{b}{d(1+\alpha)}$. It has been established in \cite[Proposition 1.8]{cdiEFC} that when $\beta=1+\alpha$, the process $(\#\Pi(t),t\geq 0)$ either comes down from infinity or stays infinite depending on $\theta<1$ or $\theta>1$,  respectively.

We study now the accessibility of the boundary $\infty$ in the case $\beta=1+\alpha$ and complete the proof of Theorem \ref{logcoal}.

\begin{lemma}\label{thirdpartTheoremslow} Let $\alpha>0$. Assume that $\mu(n)\underset{n\rightarrow \infty}{\sim}b \frac{(\log n)^{\alpha}}{n^2}$ and  $\Phi(n)\underset{n\rightarrow \infty}{\sim} d n(\log n)^{\alpha+1}.$
If $\theta>1$, then the process explodes and $\infty$ is an exit boundary. If $\theta<1$, then the process does not explode and $\infty$ is an entrance.
\end{lemma}
\begin{proof} For simplicity,  we assume that $\mu(n)=b \frac{(\log n)^{\alpha}}{n^2}$ for all $n\geq 1$. The case for which only the equivalence $\mu(n)\underset{n\rightarrow \infty}{\sim} b \frac{(\log n)^{\alpha}}{n^2}$ holds, follows from an easy adaptation.
Let $a>1$ and $\epsilon>0$. One can check that for any $0\leq x\leq \epsilon$, \[g(x)=1- (1+x)^{1-a}-(a-1)(1+\epsilon)^{-a}x\geq 0.\]
Assume $\theta>1$. For any $n\geq 1$
\begin{align*}
G_a^{f}(n)&=b n\sum_{k=1}^{\infty}\frac{(\log k)^{\alpha}}{k^2}\left[1-(1+k/n)^{1-a}\right]\geq b n \int_{1}^{\epsilon n}\frac{(\log x)^{\alpha}}{x^2}\left[1-(1+x/n)^{1-a}\right] \ddr x\\
&\geq b (a-1)(1+\epsilon)^{-a} \int_{1}^{\epsilon n}\frac{(\log x)^{\alpha}}{x}\ddr x= b (a-1)(1+\epsilon)^{-a}\frac{(\log \epsilon n)^{\alpha+1}}{\alpha+1}.
\end{align*}
By Lemma \ref{estimatesexplosion}-(i), for any $p>0$, \[-G_a^{c,p}(n)\leq d (\log n)^{\alpha+1}(a-1)(1-p)^{-a}.\]
Thus, for all $n$, $$\frac{-G_a^{c,p}(n)}{G^f_a(n)}\leq \frac{d(\alpha+1)}{b} \left(\frac{1+\epsilon}{1-p}\right)^{a}.$$
Recall that  $\frac{d(\alpha+1)}{b}=\frac{1}{\theta}<1$ and choose both $\epsilon$ and $p$ small enough such that  $\frac{d(\alpha+1)}{b}\frac{1+\epsilon}{1-p}<1$. We see that there is $a>1$ such that $\underset{n\rightarrow \infty}{\limsup}\frac{-G_a^{c,p}(n)}{G_a^{f}(n)}<1$. By Corollary \ref{genexplosionmaincor}, we see that when $\theta>1$, the process $(\#\Pi(t\wedge \sigma_p),t\geq 0)$ explodes with positive probability. Following the same argument as in the proof of Lemma \ref{lemmaexplosion}, we get that the process explodes almost-surely.

Assume now $\theta<1$. Let $a<1$. Recall $\sum_{k=1}^{n}k \mu(k)\underset{n\rightarrow }{\sim} \frac{b}{\alpha +1} (\log n)^{\alpha+1}$ and  $\Phi(n)\underset{n\rightarrow \infty}{\sim} d n(\log n)^{\alpha+1}$. We apply Proposition \ref{suffcondpropentrance1}. By comparison with an integral and by applying Karamata's theorem, we obtain \[\sum_{k=n+1}^{\infty}k^{1-a}\mu(k)\underset{n\rightarrow \infty}{\sim}  b \int_{n+1}^{\infty}x^{-a-1}(\log x)^{\alpha}\ddr x \underset{n\rightarrow \infty}{\sim} \frac{b}{a} n^{-a}(\log n)^{\alpha}.\]
Therefore, $\frac{n^{a+1}}{\Phi(n)}\sum_{k=n+1}^{\infty}k^{1-a}\mu(k)\underset{n\rightarrow \infty}{\sim} \frac{b}{d \log n}\underset{n\rightarrow \infty}{\longrightarrow} 0$, and  \eqref{secondcond} holds.
Moreover,
\[\underset{n\rightarrow \infty}{\limsup} \frac{n}{\Phi(n)} \sum_{k=1}^{n}k \mu(k)=\frac{b}{d(1+\alpha)}=\theta.\]

Applying Proposition \ref{suffcondpropentrance1}, we see that if $\theta<1$, then the process does not explode. The process comes down from infinity by Theorem \ref{cditheorem}. \qed
\end{proof}
\noindent \textbf{Proof of Theorem \ref{logcoal}.} The first two statements of Theorem \ref{logcoal} are provided by Lemma \ref{logcoallemma}. The third statement is provided by Lemma \ref{thirdpartTheoremslow}. \qed

We conclude this article by highlighting that 
the nature of the boundary $\infty$ is not known in general for the critical cases $\sigma=1$ and/or $\theta=1$. In view of the proof of Proposition \ref{criticalvalue}, this may require finer estimates that are unavailable with the functions $(G_a, a>0)$. Last, the question whether the regular boundary $\infty$ is sticky or reflecting (Theorem \ref{stablefragtheorem}) has not been treated in this paper. It is established that the boundary is regular reflecting in \cite{WFselection}. The results presented here and in \cite{cdiEFC} have counterparts 
for certain processes in duality, called $\Lambda$-Wright-Fisher processes with frequency-dependent selection. They are treated in \cite{WFselection}.

\vspace*{0.15cm}

\noindent \textbf{Acknowledgements:} C.F's research is partially  supported by LABEX MME-DII (ANR11-LBX-0023-01). X.Z.'s research is supported by Natural Sciences and Engineering Research Council of Canada (RGPIN-2016-06704) and by  National Natural Science Foundation of China (No.\  11731012).

\end{document}